\documentclass[12pt]{amsart}

\usepackage{amscd,latexsym,amsthm,amsfonts,amssymb,amsmath,amsxtra}
\usepackage[mathscr]{eucal}
\usepackage{hyperref}
\usepackage{xcolor}

\renewcommand\theequation{\thesection.\arabic{equation}}

\UseRawInputEncoding

\newcommand{\BA}{{\mathbb {A}}}

\newcommand{\BC}{{\mathbb {C}}}

\newcommand{\BQ}{{\mathbb {Q}}}
\newcommand{\BR}{{\mathbb {R}}}

\newcommand{\BZ}{{\mathbb {Z}}}

\newcommand{\CA}{{\mathcal {A}}}

\newcommand{\CF}{{\mathcal {F}}}

\newcommand{\CM}{{\mathcal {M}}}

\newcommand{\CO}{{\mathcal {O}}}

\newcommand{\Fp}{{\mathfrak {p}}}

\newcommand{\RG}{{\mathrm {G}}}

\newcommand{\RO}{{\mathrm {O}}}

\newcommand{\RU}{{\mathrm {U}}}

\newcommand{\bs}{\backslash}

\newcommand{\cusp}{{\mathrm{cusp}}}

\newcommand{\disc}{{\mathrm{disc}}}

\newcommand{\GL}{{\mathrm{GL}}}

\newcommand{\Ind}{{\mathrm{Ind}}}

\newcommand{\SL}{{\mathrm{SL}}}

\newcommand{\SO}{{\mathrm{SO}}}

\newcommand{\Sym}{{\mathrm{Sym}}}

\newcommand{\Sp}{{\mathrm{Sp}}}

\newcommand{\wt}{\widetilde}

\newcommand{\ul}{\underline}

\def\bks{{\backslash}}

\newtheorem{thm}{Theorem}[section]
\newtheorem{cor}[thm]{Corollary}
\newtheorem{lem}[thm]{Lemma}
\newtheorem{prop}[thm]{Proposition}
\newtheorem {conj}[thm]{Conjecture}

\newtheorem {ques/conj}[thm]{Question/Conjecture}

\newtheorem{defn}[thm]{Definition}
\newtheorem{rmk}[thm]{Remark}

\makeatletter

\newcommand{\Rmnum}[1]{\expandafter\@slowromancap\romannumeral #1@}
\makeatother

\begin{document}
\renewcommand{\theequation}{\arabic{equation}}
\numberwithin{equation}{section}

\title[Wave Front Sets of Global A-Packets: Upper Bound]
{On Wave Front Sets of Global Arthur Packets of Classical Groups: Upper Bound}

\author{Dihua Jiang}
\address{School of Mathematics, University of Minnesota, Minneapolis, MN 55455, USA}
\email{dhjiang@math.umn.edu}

\author{Baiying Liu}
\address{Department of Mathematics\\
Purdue University\\
West Lafayette, IN, 47907, USA}
\email{liu2053@purdue.edu}


\date{December 1st, 2021}



\thanks{The research of the  first-named author is partially supported by the NSF Grants DMS-1901802. 
The research of the  second-named author is partially supported by the NSF Grants DMS-1702218, DMS-1848058, and by start-up funds from the Department of Mathematics at Purdue University}

\begin{abstract}
We prove a conjecture of the first-named author (\cite{J14}) on the upper bound Fourier coefficients of automorphic forms in Arthur packets of all classical groups over any number field. 
This conjecture generalizes the global version of the local tempered $L$-packet conjecture of F. Shahidi (\cite{Sh90} and \cite{Sh10}).
\end{abstract}

\maketitle


\section{Introduction}
In the classical theory of automorphic forms, Fourier coefficients encode abundant arithmetic information of automorphic forms. In the modern theory of automorphic forms, i.e. the theory of automorphic representations of reductive algebraic groups defined over a number field $k$ (or a global field), Fourier coefficients bridge the connection from harmonic analysis to number theory via automorphic forms.
When the reductive group is $\GL_n$, the general linear group, by a classical theorem of I. Piatetski-Shapiro (\cite{PS79}) and J. Shalika  (\cite{S74}), every cuspidal automorphic representation of $\GL_n(\BA)$, where $\BA$ is the ring of adeles of $k$, has a non-zero Whittaker-Fourier coefficient. This fundamental result has been indispensable in the theory, 
especially the theory of automorphic $L$-functions. 
The theorem of Piatetski-Shapiro and Shalika has been extended to the discrete spectrum of $\GL_n(\BA)$ in \cite{JL13} and to the isobaric sum automorphic spectrum of $\GL_n(\BA)$ in \cite{LX21}. 

In general, due to the nature of the discrete spectrum of square-integrable automorphic forms on reductive algebraic groups $G$, one has to consider more general version of Fourier coefficients, i.e. Fourier coefficients of automorphic forms associated to nilpotent orbits in the Lie algebra $\frak{g}$ of $G$. Such general Fourier coefficients of automorphic forms, including Bessel-Fourier coefficients and Fourier-Jacobi coefficients have been
widely used in theory of automorphic $L$-functions via integral representation method (see \cite{GPSR97}, \cite{GJRS11}, \cite{JZ14} and \cite{JZ20}, for instance),
in automorphic descent method of Ginzburg, Rallis and Soudry to produce special cases of explicit Langlands functorial transfers (\cite{GRS11}),
and in the Gan-Gross-Prasad conjecture on vanishing of the central value of certain automorphic $L$-functions of symplectic type
(\cite{GJR04}, \cite{JZ20}, and \cite{GGP12}). More recent applications of such general Fourier coefficients to explicit constructions of endoscopy
transfers for classical groups can be found in \cite{J14} (and also in \cite{G12} for split classical groups).


In this paper, we consider following classical groups defined over $k$, $G_n=\Sp_{2n}, \SO_{2n+1}, \RO_{2n}^{\alpha}$, quasi-split, 
and $\RU_n$, quasi-split or inner forms. 
We follow the formulation in \cite{GGS17} for the definition of generalized Whittaker-Fourier coefficients of automorphic forms associated to nilpotent orbits, see Section 2 for details. 
It is well-known that nilpotent orbits of the quasi-split classical group $G_n$ are parameterized by symplectic or orthogonal partitions and certain quadratic forms when $G_n=\Sp_{2n}, \SO_{2n+1}, \RO_{2n}^{\alpha}$, by relevant partitions when $G_n=\RU_n$ (see \cite{CM93}, \cite{N11} and \cite{W01}, for instance). 
For any irreducible automorphic representation $\pi$
of $G_n(\BA)$, let $\mathfrak{n}(\pi)$ be the set of nilpotent orbits providing nonzero generalized Whittaker-Fourier coefficients for $\pi$, which is called the {\it wave front set} of $\pi$, 
as in \cite{JLS16} for instance. Let $\mathfrak{n}^m(\pi)$ be the subset that consists of maximal elements in $\mathfrak{n}(\pi)$ under the dominance ordering of nilpotent orbits, and denote by $\frak{p}^m(\pi)$ the set of the partitions of type $G_n$ corresponding to nilpotent orbits in $\mathfrak{n}^m(\pi)$. 

It is an interesting problem to determine the structure of the set $\mathfrak{n}^m(\pi)$ and equivalently the set $\frak{p}^m(\pi)$ for any given irreducible automorphic
representation $\pi$ of $G_n(\BA)$, by means of other invariants of $\pi$. When $\pi$ occurs in the discrete spectrum of square integrable automorphic functions on
$G_n(\BA)$, the global Arthur parameter attached to $\pi$ (\cite{Ar13}, \cite{Mok15}, \cite{KMSW14}) is clearly a fundamental invariant for $\pi$.
An important conjecture made in \cite{J14}, which is the natural generalization of the global version of the local tempered $L$-packet conjecture of F. Shahidi (\cite{Sh90} and \cite{Sh10}), asserts an intrinsic relation between the structure of the global Arthur parameter of $\pi$ to the structure of the set $\frak{p}^m(\pi)$. It is well-known that the conjecture of Shahidi and its global version 
(see \cite[Section 3]{JL16a} for discussion and proof) has played a fundamental role in the understanding of the local and global Arthur packets for generic Arthur parameters, according the endoscopic classification of J. Arthur (\cite{Ar13} and also \cite{Mok15} and \cite{KMSW14}).  It is well expected that the conjecture made in \cite{J14} for general global Arthur parameters will be important 
to the understanding of the structure of general global Arthur packets. 

To state the conjecture of \cite{J14}, for simplicity, we briefly recall the endoscopic classification
of the discrete spectrum for $G_n(\BA)$ from \cite{Ar13} for $G_n=\Sp_{2n}, \SO_{2n+1}, \RO_{2n}^{\alpha}$.

The set of global Arthur parameters for the discrete spectrum of $G_n$ is denoted, as in \cite{Ar13},
by $\wt{\Psi}_2(G_n)$, the elements of which are of the form
\begin{equation}\label{psin}
\psi:=\psi_1\boxplus\psi_2\boxplus\cdots\boxplus\psi_r,
\end{equation}
where $\psi_i$ are pairwise different simple global Arthur parameters of orthogonal type (when $G_n=\Sp_{2n}, \RO_{2n}^{\alpha}$) or symplectic type (when $G_n=\SO_{2n+1}$), 
and have the form $\psi_i=(\tau_i,b_i)$. The notations are explained in order. 
Let $\CA_\cusp(\GL_{a_i})$ be the set of equivalence classes of irreducible cuspidal automorphic representations of $\GL_{a_i}(\BA)$. 
We have $\tau_i\in\CA_\cusp(\GL_{a_i})$ with 
\[\sum_{i=1}^r a_ib_i
=
\begin{cases}
2n+1,&\quad {\rm when}\  G_n=\Sp_{2n};\\
2n,&\quad {\rm when}\  G_n=\SO_{2n+1}\ {\rm or}\ \RO_{2n}^{\alpha},
\end{cases}
\]
and the central character of $\psi$ and the central characters of $\tau_i$'s satisfy the following constraints:
\[
\prod_i \omega_{\tau_i}^{b_i} = 
\begin{cases}
1,&\quad {\rm when}\ G_n=\Sp_{2n}\ {\rm or}\ \SO_{2n+1};\\
 \eta_{\alpha},&\quad {\rm when}\ G_n=\RO_{2n}^{\alpha},
\end{cases}
\]
following \cite[Section 1.4]{Ar13}. More precisely, for each $1 \leq i \leq r$, $\psi_i=(\tau_i,b_i)$
satisfies the following conditions:
if $\tau_i$ is of symplectic type (i.e., $L(s, \tau_i, \wedge^2)$ has a pole at $s=1$), then $b_i$ is even (when $G_n=\Sp_{2n}, \RO_{2n}^{\alpha}$), or odd (when $G_n=\SO_{2n+1}$); and 
if $\tau_i$ is of orthogonal type (i.e., $L(s, \tau_i, \Sym^2)$ has a pole at $s=1$), then $b_i$ is odd (when $G_n=\Sp_{2n}, \RO_{2n}^{\alpha}$), or even (when $G_n=\SO_{2n+1}$). Given a global Arthur parameter $\psi$ as above, recall from \cite{J14} that $\ul{p}(\psi)=[(b_1)^{a_1} \cdots (b_r)^{a_r}]$ is the partition attached to $(\psi, G^{\vee}(\BC))$.

\begin{thm}[Theorem 1.5.2, \cite{Ar13}]\label{thm-Arthur}
For each global Arthur parameter $\psi\in\wt{\Psi}_2(G_n)$ a global Arthur
packet $\wt{\Pi}_\psi$ is defined. The discrete spectrum of $G_n(\BA)$ has the following decomposition
$$
L^2_\disc(G_n(k)\bks G_n(\BA))
\cong\oplus_{\psi\in\wt{\Psi}_2(G_{n})}m_\psi\left(\oplus_{\pi\in\wt{\Pi}_\psi(\epsilon_\psi)}\pi\right),
$$
where $\wt{\Pi}_\psi(\epsilon_\psi)$ denotes the subset of $\wt{\Pi}_\psi$ consisting of members which occur in the discrete spectrum, and $m_\psi$ is the discrete multiplicity of $\Pi$, which 
is either $1$ or $2$.
\end{thm}

As in \cite{J14}, one may call $\wt{\Pi}_\psi(\epsilon_\psi)$ the automorphic $L^2$-packet attached to $\psi$.
For $\pi\in\wt{\Pi}_\psi(\epsilon_\psi)$, the structure of the global Arthur parameter $\psi$ deduces constraints on
the structure of $\frak{p}^m(\pi)$, which is given by the following conjecture of the first-named author. 

\begin{conj}[Conjecture 4.2, Parts (1) and (2), \cite{J14}]\label{cubmfc}
For any $\psi\in\wt{\Psi}_2(G_n)$, let $\wt{\Pi}_{\psi}(\epsilon_\psi)$ be the automorphic $L^2$-packet attached to $\psi$.
Assume that $\underline{p}(\psi)$ is the partition attached to $(\psi,G^\vee(\BC))$. For any  $\pi\in\wt{\Pi}_{\psi}(\epsilon_\psi)$, if a partition $\ul{p}\in\frak{p}^m(\pi)$, then 
$$\ul{p}\leq \eta_{{\frak{g}^\vee,\frak{g}}}(\ul{p}(\psi)).$$
Here $\eta_{{\frak{g}^\vee,\frak{g}}}$ denotes the Barbasch-Vogan-Spaltenstein duality map from the partitions for the dual group $G^\vee(\BC)$ to
the partitions for $G$ as introduced in \cite{Sp82} and \cite{BV85}, and see also \cite{Ac03}.   
\end{conj}

Conjecture 4.2 in \cite{J14} consists of the upper-bound conjecture (Conjecture \ref{cubmfc}) and the sharpness conjecture (\cite[Conjecture 4.2, Part (3)]{J14}, i.e., there exists $\pi\in\wt{\Pi}_{\psi}(\epsilon_\psi)$ such that $\eta_{{\frak{g}^\vee,\frak{g}}}(\ul{p}(\psi)) \in \frak{p}^m(\pi)$). 
It is clear that if the global Arthur parameter $\psi$ is generic, then \cite[Conjecture 4.2]{J14} asserts that the corresponding global Arthur packet $\wt{\Pi}_{\psi}(\epsilon_\psi)$ contains a 
automorphic member that is generic, i.e. has a non-zero Whittaker-Fourier coefficient. This is the global version of the local tempered $L$-packet conjecture of Shahidi (\cite{Sh90}) and was proved in 
\cite[Section 3]{JL16a} by using automorphic descent of D. Ginzgurg, S. Rallis, and D. Soudry 
(\cite{GRS11}). The goal of this paper is to prove Conjecture \ref{cubmfc} for general global Arthur parameters. The sharpness conjecture is of global in nature and will be fully considered in future projects. 

In \cite{JL16}, using the method of local descent, we partially prove Conjecture \ref{cubmfc} for $G_n=\Sp_{2n}$, namely, for any  $\pi\in\wt{\Pi}_{\psi}(\epsilon_\psi)$, if a partition $\ul{p}\in\frak{p}^m(\pi)$, then 
$$\ul{p}\leq_L \eta_{{\frak{g}^\vee,\frak{g}}}(\ul{p}(\psi)),$$
under the lexicographical order. 
We refer to \cite[Section 4]{J14} for more discussion on this conjecture and related topics. 

In order to prove Conjecture \ref{cubmfc}, we study the structure of the unramified local components $\pi_v$ of $\pi$ and of the set $\frak{p}^m(\pi_v)$ which is defined similarly as $\frak{p}^m(\pi)$. 
Our discussion reduces the general situation to a special case of strongly negative 
unramified unitary representations of $G_n$ (see Section 3 for details). In 
such a special situation, the structure of the wave front set 
(Theorem \ref{keylemma0}) can be deduced as a special case from \cite[Theorem 1.5]{Oka21}. 

To be more precise, first, for the Arthur parameter $\psi=\boxplus_{i=1}^r (\tau_i,b_i)$, by \cite[Proposition 6.1]{JL16} (see Proposition \ref{propsq}), there exist infinitely many finite places $v$ such that $G_n(k_v)$ is split, $\tau_{i,v}$'s all have trivial central characters, and hence $\pi_v$ is the unramified component of an induced representation of the following form
$$\sigma=\times_{i=1}^r v^{\alpha_i} \chi_i({\det}_{m_i}) \rtimes \sigma_{sn},$$
where $0 \leq \alpha_i < 1$, 
$\sigma_{sn}$ is a special family of strongly negative representations which have Arthur parameters
of the form $\oplus_{j=1}^s 1_{W_F'} \otimes S_{2n_j+1}$ (see Section 3.2 for details), with $W_F'$ being the Weil-Deligne group and $n_1< n_2 < \cdots < n_s$. It is known that the wave front set of $\sigma$ (hence of $\pi_v$) is bounded above by the induced orbits once we know the leading orbits for the wave front set of $\sigma_{sn}$. On the other hand, Okada (\cite[Theorem 1.5]{Oka21} computed the leading orbits in the wave front set of those unramified representations whose Arthur parameters are trivial when restricting to the Weil-Deligne group. 

\begin{thm}[Main Theorem]\label{main}
Conjecture \ref{cubmfc} holds for any $\psi\in\wt{\Psi}_2(G_n)$.
\end{thm}

We remark that for non-quasi-split even orthogonal groups, once the Arthur classification being carried out (see \cite{CZ21a, CZ21b} for recent progress in this direction), Conjecture \ref{cubmfc} can be proved by similar arguments. 

In the last part of this paper, we study the wave front set of the unramified unitary dual for split classical groups $G_n=\Sp_{2n}, \SO_{2n+1}, \RO_{2n}$. Under a conjecture on the wave front set of negative representations (Conjecture \ref{conj-neg}), we are able to determine the set $\frak{p}^m(\pi)$ for general unramified unitary representations (Theorem \ref{main3}). This provides a reduction towards understanding the wave front set of the whole unramified unitary dual, which has its own interests. 

The structure of this paper is as follows. In Section 2, we recall certain twisted Jacquet modules and Fourier coefficients associated to nilpotent orbits, following the formulation in \cite{GGS17}. The structure of unramified unitary dual of $G_n(F_v)$
was determined by D. Barbasch in \cite{Bar10} and by G. Muic and M. Tadic in \cite{MT11} with different approaches. In Section 3, we recall
from \cite{MT11} the results on unramified unitary dual. In Section 4, we determine, for any given global Arthur parameter $\psi\in\wt{\Psi}_2(G_n)$,
the unramified components $\pi_v$ of any $\pi\in\wt{\Pi}_{\psi}(\epsilon_\psi)$ in terms of the classification data in \cite{MT11}, and prove Theorem \ref{main} by means of Theorem \ref{main2} which is about certain properties of $\ul{p} \in \frak{p}^m(\pi_v)$. 
Theorem \ref{main2} is technical and will be proved in Sections 5, 6, and 7, for $G_n=\Sp_{2n}, \SO_{2n+1}, \RO_{2n}^{\alpha}$, respectively. In Section 8, we determine the leading orbits in the wave front set of general unramified unitary representations assuming Conjecture \ref{conj-neg} for split classical groups $G_n=\Sp_{2n}, \SO_{2n+1}, \RO_{2n}$ (Theorem \ref{main3}). 

\subsection*{Acknowledgements} 
We would like to thank Freydoon Shahidi, Fan Gao, and Lei Zhang for helpful comments and suggestions. 


\section{Fourier coefficients associated to nilpotent orbits}\label{def of FC's}


In this section, we recall certain twisted Jacquet modules and Fourier coefficients associated to nilpotent orbits, following the formulation of R. Gomez, D. Gourevitch and S. Sahi in \cite{GGS17}.

Let $\RG$ be a reductive group defined over a field $F$ of characteristic zero, and $\mathfrak{g}$ be the Lie algebra of $G=\RG(F)$. 
Given any semi-simple element $s \in \mathfrak{g}$, under the adjoint action, $\mathfrak{g}$ is decomposed into a direct sum of eigenspaces $\mathfrak{g}^s_i$ corresponding to eigenvalues $i$.
The element
$s$ is called {\it rational semi-simple} if all its eigenvalues are in $\BQ$.
Given a nilpotent element $u$ and a semi-simple element $s$ in $\mathfrak{g}$, the pair $(s,u)$ is called a {\it Whittaker pair} if $s$ is a rational semi-simple element, and $u \in \mathfrak{g}^s_{-2}$. The element $s$ in a Whittaker pair $(s, u)$ is called a {\it neutral element} for $u$ if there is a nilpotent element $v \in \mathfrak{g}$ such that $(v,s,u)$ is an $\mathfrak{sl}_2$-triple. A Whittaker pair $(s, u)$ with $s$ being a neutral element is called a {\it neutral pair}. 

Given any Whittaker pair $(s,u)$, define an anti-symmetric form $\omega_u$ on $\mathfrak{g}\times \mathfrak{g}$ by 
\begin{align}\label{form}
 \omega_u(X,Y):=\kappa(u,[X,Y]),   
\end{align}
here $\kappa$ is the Killing form on $\mathfrak{g}$.
For any rational number $r \in \BQ$, let $\mathfrak{g}^s_{\geq r} = \oplus_{r' \geq r} \mathfrak{g}^s_{r'}$. 
 Let $\mathfrak{u}_s= \mathfrak{g}^s_{\geq 1}$ and let $\mathfrak{n}_{s,u}$ be the radical of $\omega_u |_{\mathfrak{u}_s}$. Then $[\mathfrak{u}_s, \mathfrak{u}_s] \subset \mathfrak{g}^s_{\geq 2} \subset \mathfrak{n}_{s,u}$. 
For any $X \in \mathfrak{g}$, let $\mathfrak{g}_X$ be the centralizer of $X$ in $\mathfrak{g}$.  
By \cite[Lemma 3.2.6]{GGS17}, one has $\mathfrak{n}_{s,u} = \mathfrak{g}^s_{\geq 2} + \mathfrak{g}^s_1 \cap \mathfrak{g}_u$.
Note that if the Whittaker pair $(s,u)$ comes from an $\mathfrak{sl}_2$-triple $(v,s,u)$, then $\mathfrak{n}_{s,u}=\mathfrak{g}^s_{\geq 2}$. We denote by 
$N_{s,u}=\exp(\mathfrak{n}_{s,u})$ the corresponding unipotent subgroup of $G$. 

When $F=k_v$ is a non-Archimedean local field, we take $\psi: F \rightarrow \BC^{\times}$ to be a fixed non-trivial additive character and define a character of $N_{s,u}$ by 
\begin{align}\label{psi-local}
    \psi_u(n)=\psi(\kappa(u,\log(n))).
\end{align}
Let $\sigma$ be an irreducible admissible representation of $G(F)$. 
The {\it twisted Jacquet module} of $\pi$ associated to
a Whittaker pair $(s,u)$ is defined to be  
$\sigma_{N_{s,u}, \psi_u}$.
Let $\mathfrak{n}(\sigma)$ be the set of nilpotent orbits $\CO\subset \mathfrak{g}$ such that the twisted Jacquet module $\sigma_{N_{s,u}, \psi_u}$ is non-zero for some neutral pair $(s,u)$ with $u \in \CO$. 

When $F=k$ is a number field, let $\BA$ be the ring of adeles, and let $\psi: F \bs \BA \rightarrow \BC^{\times}$ be a fixed non-trivial additive character.
Extend the killing form $\kappa$ to $\mathfrak{g}(\BA)\times \mathfrak{g}(\BA)$.
Define a character of $N_{s,u}(\BA)$ by 
\begin{align}\label{psi-global}
    \psi_u(n)=\psi(\kappa(u,\log(n))).
\end{align}
It is clear from the definition that the character $\psi_u(n)$ is trivial when restricted to the 
discrete subgroup $N_{s,u}(F)$, and hence can be viewed as a function on 
$[N_{s,u}]:=N_{s,u}(F)\bs N_{s,u}(\BA)$.
Let $\pi$ be an irreducible automorphic representation of $\RG(\BA)$. For any $\phi \in \pi$, the {\it degenerate Whittaker-Fourier coefficient} of $\phi$ attached to 
a Whittaker pair $(s,u)$ is defined to be
\begin{equation}\label{dwfc}
\CF_{s,u}(\phi)(g):=\int_{[N_{s,u}]} \phi(ng) \psi_u^{-1}(n) \, \mathrm{d}n\,.
\end{equation}
If $(s,u)$ is a neutral pair, then $\CF_{s,u}(\phi)$ is also called a {\it generalized Whittaker-Fourier coefficient} of $\phi$.
Define
\begin{align}\label{FC-pi}
    \CF_{s,u}(\pi):=\left\{\CF_{s,u}(\phi)\ |\ \phi \in \pi \right\},
\end{align}
which is called the Fourier coefficient of $\pi$. 
The {\it wave-front set} $\mathfrak{n}(\pi)$ of $\pi$ is defined to be the set of nilpotent orbits $\CO$ such that $\CF_{s,u}(\pi)$ is non-zero for some neutral pair $(s,u)$ with $u \in \CO$. 

Note that if $\sigma_{N_{s,u}, \psi_u}$, or $\CF_{s,u}(\pi)$, is non-zero for some neutral pair $(s,u)$ with $u \in \CO$, then it is non-zero for any such neutral pair $(s,u)$, since the non-vanishing property of such Whittaker models or Fourier coefficients does not depend on the choices of representatives of $\CO$.
Moreover, we let $\mathfrak{n}^m(\sigma)$ and $\mathfrak{n}^m(\pi)$ be the set of maximal elements in the wave front sets $\mathfrak{n}(\sigma)$ and $\mathfrak{n}(\pi)$, respectively,  under the natural ordering of nilpotent orbits (i.e., $\CO_1 \leq \CO_2$ if $\CO_1 \subset \overline{\CO_2}$, the Zariski closure of $\CO_2$).

In this paper, we mainly consider classical groups 
$G_n=\Sp_{2n}, \SO_{2n+1}$, $\RO_{2n}^{\alpha}$, quasi-split, 
and $\RU_n$, quasi-split or inner forms, 
and study the sets
$\frak{p}(\pi)$, or $\frak{p}(\sigma)$, partitions corresponding to the nilpotent orbits in the wave front sets,
for any irreducible automorphic representation $\pi$ of $G_n(\BA)$, which occurs in the discrete spectrum of $G_n(\BA)$ as displayed in Theorem \ref{thm-Arthur}, and for some unramified local components 
$\sigma=\pi_v$.


\section{Unramified Unitary Dual of split classical groups}\label{unramified unitary dual}

We take $F=k_v$ to be a non-Archimedean local field of $k$. 
In this section, we recall the classification of the unramified unitary dual of 
the split classical groups $G_n=\Sp_{2n}, \SO_{2n+1}, \RO_{2n}$ over $F$, which was obtained by
D. Barbasch in \cite{Bar10} and by G. Muic and M. Tadic in \cite{MT11}, using different methods. 

We recall the unramified unitary dual of split classical groups from the work of 
Muic and Tadic in \cite{MT11}. 
The classification in \cite{MT11} starts from classifying two special families of irreducible unramified representations of $G_n(F)$ that are called {\bf strongly negative} and {\bf negative}, respectively. We refer to \cite{M06} for definitions of strongly negative and negative
representations, respectively, and for more related discussion on those two families of unramified representations.
In the following, we recall from \cite{MT11} the classification of these two families in terms of {\bf Jordan blocks}. Their classification also provides explicit constructions of the two families of unramified representations.

A pair $(\chi, m)$, where $\chi$ is an unramified unitary
character of $F^*$ and $m \in \BZ_{>0}$, is called
a {\bf Jordan block}.
When $G_n=\Sp_{2n}, \RO_{2n}$, define $\rm{Jord}_{sn}(n)$ to be the collection of all sets
$\rm{Jord}$ of the following form:
\begin{equation}
\label{sec7equ1sp}
\{(\lambda_0, 2n_1+1), \ldots, (\lambda_0, 2n_k+1), (1_{\GL_1}, 2m_1+1),
\ldots, (1_{\GL_1}, 2m_l+1)\}
\end{equation}
where $\lambda_0$ is the unique non-trivial unramified
unitary character of $F^*$ of order 2,
given by the local Hilbert symbol $(\delta, \cdot)_{F^*}$, with $\delta$ being a non-square unit in $\CO_{F}$;
$k$ is even, and $l$ is odd when $G_n=\Sp_{2n}$ and is even when $G_n=\RO_{2n}$. There are also following constraints:
$$
0 < n_1 < n_2 < \cdots < n_k, \quad 0 < m_1 < m_2 < \cdots < m_l;
$$
and 
\[
\sum_{i=1}^k (2n_i+1) + \sum_{j=1}^l (2m_j+1) = 
\begin{cases}
2n+1&{\rm when}\ G_n=\Sp_{2n}\\
2n&{\rm when}\ G_n=\RO_{2n}.
\end{cases}
\]
When $G_n=\SO_{2n+1}$, define $\rm{Jord}_{sn}(n)$ to be the collection of all sets
$\rm{Jord}$ of the following form:
\begin{equation}
\label{sec7equ1so}
\{(\lambda_0, 2n_1), \ldots, (\lambda_0, 2n_k), (1_{\GL_1}, 2m_1),
\ldots, (1_{\GL_1}, 2m_l)\}
\end{equation}
where
$$
0 \leq n_1 < n_2 < \cdots < n_k, \ \ \ \ 0 \leq m_1 < m_2 < \cdots < m_l,
$$
both $k$ and $l$ are even 
and 
$\sum_{i=1}^k (2n_i) + \sum_{j=1}^l (2m_j) = 2n$. 

For each $\rm{Jord} \in \rm{Jord}_{sn}(n)$, we can associate a
representation $\sigma(\rm{Jord})$, which is the unique irreducible
unramified subquotient of the following induced
representation. When $G_n=\Sp_{2n}$, it is given by 
\begin{align}\label{sec7equ2-1}
\begin{split}
 & \nu^{\frac{n_{k-1}-n_k}{2}} \lambda_0({\det}_{n_{k-1}+n_k+1})
\times
\nu^{\frac{n_{k-3}-n_{k-2}}{2}} \lambda_0({\det}_{n_{k-3}+n_{k-2}+1})\\
&\qquad\times \cdots \times
\nu^{\frac{n_{1}-n_2}{2}} \lambda_0({\det}_{n_{1}+n_2+1})\\
&\qquad \times
\nu^{\frac{m_{l-1}-m_l}{2}} 1_{{\det}_{m_{l-1}+m_l+1}}
\times
\nu^{\frac{m_{l-3}-m_{l-2}}{2}} 1_{{\det}_{m_{l-3}+m_{l-2}+1}}\\
& \qquad\times \cdots \times
\nu^{\frac{m_{2}-m_3}{2}} 1_{{\det}_{m_{2}+m_3+1}}
\rtimes 1_{\Sp_{2m_1}}. 
\end{split}
\end{align}
When $G_n=\RO_{2n}$, it is given by 
\begin{align}\label{sec7equ2-2}
\begin{split}
 & \nu^{\frac{n_{k-1}-n_k}{2}} \lambda_0({\det}_{n_{k-1}+n_k+1})
\times
\nu^{\frac{n_{k-3}-n_{k-2}}{2}} \lambda_0({\det}_{n_{k-3}+n_{k-2}+1})\\
&\qquad \times \cdots \times
\nu^{\frac{n_{1}-n_2}{2}} \lambda_0({\det}_{n_{1}+n_2+1})\\
& \qquad\times
\nu^{\frac{m_{l-1}-m_l}{2}} 1_{{\det}_{m_{l-1}+m_l+1}}
\times
\nu^{\frac{m_{l-3}-m_{l-2}}{2}} 1_{{\det}_{m_{l-3}+m_{l-2}+1}}\\
&\qquad \times \cdots \times
\nu^{\frac{m_{1}-m_2}{2}} 1_{{\det}_{m_{1}+m_2+1}}
\rtimes 1_{\RO_0}. 
\end{split}
\end{align}
When $G_n=\SO_{2n+1}$, it is given by 
\begin{align}\label{sec7equ2-3}
\begin{split}
& \nu^{\frac{n_{k-1}-n_k}{2}} \lambda_0({\det}_{n_{k-1}+n_k})
\times
\nu^{\frac{n_{k-3}-n_{k-2}}{2}} \lambda_0({\det}_{n_{k-3}+n_{k-2}})\\
& \qquad\times \cdots \times
\nu^{\frac{n_{1}-n_2}{2}} \lambda_0({\det}_{n_{1}+n_2})\\
&\qquad \times
\nu^{\frac{m_{l-1}-m_l}{2}} 1_{{\det}_{m_{l-1}+m_l}}
\times
\nu^{\frac{m_{l-3}-m_{l-2}}{2}} 1_{{\det}_{m_{l-3}+m_{l-2}}}\\
&\qquad \times \cdots \times
\nu^{\frac{m_{1}-m_2}{2}} 1_{{\det}_{m_{1}+m_2}}
\rtimes 1_{\SO_1}.
\end{split}
\end{align}

\begin{thm}[Theorem 5-8, \cite{MT11}] \label{thm8}
Assume that $n > 0$. The map $\rm{Jord} \mapsto \sigma(\rm{Jord})$ defines a
one-to-one correspondence between the set
$\rm{Jord}_{sn}(n)$ to the set of all
irreducible strongly negative
unramified representations of $G_n(F)$.
\end{thm}

The inverse of the map in Theorem \ref{thm8} is denoted
by $\sigma \mapsto \rm{Jord}(\sigma)$. Based on the classification in Theorem \ref{thm8}, 
irreducible negative unramified representations can be constructed from
irreducible
strongly negative unramified representations of smaller rank groups as follows.

\begin{thm}[Thereom 5-10, \cite{MT11}]\label{thm9}
For any sequence of pairs $(\chi_1, n_1), \ldots, (\chi_t, n_t)$ with
$\chi_i$ being unramified unitary characters of $F^*$ and
$n_i \in \BZ_{\geq 1}$, for $1 \leq i \leq t$,
and for a strongly negative representation $\sigma_{sn}$ of $G_{n'}(F)$
with $\sum_{i=1}^t n_i + n' = n$,
the unique irreducible unramified subquotient
of the following induced representation
\begin{equation}\label{thm9equ1}
\chi_1 ({\det}_{{n_1}}) \times \cdots \times \chi_t ({\det}_{{n_t}}) \rtimes \sigma_{sn}
\end{equation}
is negative and it is a subrepresentation.

Conversely, any irreducible negative
unramified representation
$\sigma_{neg}$ of $G_{n}(F)$ can be obtained from the above construction.
The data 
\[
(\chi_1, n_1), \ldots, (\chi_t, n_t)
\]
and $\sigma_{sn}$
are unique, up to permutations and taking inverses of $\chi_i$'s.
\end{thm}

For any irreducible negative unramified representation $\sigma_{neg}$
with data in Theorem \ref{thm9}, we define
\begin{equation*}
\rm{Jord}(\sigma_{neg}) = \rm{Jord}(\sigma_{sn}) \cup \{(\chi_i,n_i),
(\chi_i^{-1},n_i) | 1 \leq i \leq t\}.
\end{equation*}
By Corollary 3.8 of \cite{M07},
any irreducible negative representation is unitary. In
particular, we have the following
\begin{cor}\label{cor:unitary}
Any irreducible negative unramified representation of $G_{n}(F)$ is
unitary.
\end{cor}

To describe the general unramified unitary dual, we need to recall the following definition.

\begin{defn}[Definition 5-13, \cite{MT11}] \label{def1}
Let $\CM^{unr}(n)$ be the set of pairs
$(\textbf{e}, \sigma_{neg})$, where $\textbf{e}$ is a
multiset of triples $(\chi, m, \alpha)$ with $\chi$ being an unramified
unitary character of $F^*$, $m \in \BZ_{>0}$ and $\alpha \in \BR_{>0}$, and
$\sigma_{neg}$ is an irreducible negative unramified
representation of $G_{n''}(F)$, having the property that $\sum_{(\chi, m)} m \cdot \#{\textbf{e}(\chi,m)} + n'' = n$
with $\textbf{e}(\chi,m) = \{\alpha | (\chi, m, \alpha) \in \textbf{e}\}$.
Note that $\alpha \in \textbf{e}(\chi,m)$ is counted with multiplicity.

Let $\CM^{u,unr}(n)$ be the subset of $\CM^{unr}(n)$ consisting of pairs
$(\textbf{e}, \sigma_{neg})$, which satisfy the following conditions:
\begin{enumerate}
\item If $\chi^2 \neq 1_{\GL_1}$, then $\textbf{e}(\chi,m) = \textbf{e}(\chi^{-1},m)$,
and $0 < \alpha < \frac{1}{2}$, for all $\alpha \in \textbf{e}(\chi,m)$.
\item If $\chi^2 = 1_{\GL_1}$, and $m$ is even, then
$0 < \alpha < \frac{1}{2}$, for all $\alpha \in \textbf{e}(\chi,m)$, when $G_n=\Sp_{2n}, \RO_{2n}$; $0 < \alpha < 1$, for all $\alpha \in \textbf{e}(\chi,m)$, when $G_n=\SO_{2n+1}$.
\item If $\chi^2 = 1_{\GL_1}$, and $m$ is odd, then
$0 < \alpha < 1$, for all $\alpha \in \textbf{e}(\chi,m)$, when $G_n=\Sp_{2n}, \RO_{2n}$; $0 < \alpha < \frac{1}{2}$, for all $\alpha \in \textbf{e}(\chi,m)$, when $G_n=\SO_{2n+1}$.
\end{enumerate}
Write elements in $\textbf{e}(\chi,m)$ as follows:
\begin{equation*}
0 < \alpha_1 \leq \cdots \leq \alpha_k \leq \frac{1}{2} < \beta_1 \leq
\cdots \leq \beta_l < 1,
\end{equation*}
with $k,l\in \BZ_{\geq 0}$. They satisfy the following conditions:
\begin{itemize}
\item[(a)] If $(\chi, m) \notin \rm{Jord}(\sigma_{neg})$,
then $k+l$ is even.
\item[(b)] If $k \geq 2$, $\alpha_{k-1} \neq \frac{1}{2}$.
\item[(c)] If $l \geq 2$, then $\beta_1 < \beta_2 < \cdots < \beta_l$.
\item[(d)] $\alpha_i + \beta_j \neq 1$, for any $1 \leq i \leq k$,
$1 \leq j \leq l$.
\item[(e)] If $l \geq 1$, then $\#\{i | 1-\beta_1 < \alpha_i  \leq \frac{1}{2}\}$
is even.
\item[(f)] If $l \geq 2$, then $\#\{i | 1-\beta_{j+1} < \alpha_i < \beta_j\}$
is odd, for any $1 \leq j \leq l-1$.
\end{itemize}
\end{defn}

\begin{thm}[Theorem 5-14, \cite{MT11}]\label{thm10}
The map
$$(\textbf{e}, \sigma_{neg}) \mapsto \times_{(\chi, m, \alpha) \in \textbf{e}}
v^{\alpha} \chi({\det}_m) \rtimes \sigma_{neg}$$
defines a one-to-one correspondence between
the set $\CM^{u,unr}(n)$ and the set of equivalence classes of all
irreducible unramified unitary representations of $G_{n}(F)$.
\end{thm}

In Sections 4-7, we will mainly consider the following type of unramified unitary representations:

\textbf{Type (I)}: An irreducible unramified unitary representations of $G_{n}(F)$ is called of {\bf Type (I)} if it is of the following form:
\begin{equation}\label{typeA}
\sigma = \times_{(\chi, m, \alpha) \in \textbf{e}}
v^{\alpha} \chi({\det}_m) \rtimes \sigma_{neg} \leftrightarrow (\textbf{e}, \sigma_{neg}),
\end{equation}
where $\sigma_{neg}$ is the unique irreducible negative unramified  subrepresentation
of the following induced representation
\begin{equation*}
\chi_1 ({\det}_{{n_1}}) \times \cdots \times \chi_t ({\det}_{{n_t}}) \rtimes \sigma_{sn},
\end{equation*}
with $\sigma_{sn}$ being the unique strongly negative unramified constituent of the following induced representation:
\begin{align*}
\begin{split}
G_n=\Sp_{2n}:\ & 
\nu^{\frac{m_{l-1}-m_l}{2}} 1_{{\det}_{m_{l-1}+m_l+1}}
\times
\nu^{\frac{m_{l-3}-m_{l-2}}{2}} 1_{{\det}_{m_{l-3}+m_{l-2}+1}}\\
& \qquad\times \cdots \times
\nu^{\frac{m_{2}-m_3}{2}} 1_{{\det}_{m_{2}+m_3+1}}
\rtimes 1_{\Sp_{2m_1}};\\
G_n=\RO_{2n}:\ & 
\nu^{\frac{m_{l-1}-m_l}{2}} 1_{{\det}_{m_{l-1}+m_l+1}}
\times
\nu^{\frac{m_{l-3}-m_{l-2}}{2}} 1_{{\det}_{m_{l-3}+m_{l-2}+1}}\\
& \qquad\times \cdots \times
\nu^{\frac{m_{1}-m_2}{2}} 1_{{\det}_{m_{1}+m_2+1}}
\rtimes 1_{\RO_0};\\
G_n=\SO_{2n+1}:\ & 
\nu^{\frac{m_{l-1}-m_l}{2}} 1_{{\det}_{m_{l-1}+m_l}}
\times
\nu^{\frac{m_{l-3}-m_{l-2}}{2}} 1_{{\det}_{m_{l-3}+m_{l-2}}}\\
& \qquad\times \cdots \times
\nu^{\frac{m_{1}-m_2}{2}} 1_{{\det}_{m_{1}+m_2}}
\rtimes 1_{\SO_1}.
\end{split}
\end{align*}

Next, we recall the following theorem, which is a special case of \cite[Theorem 1.5]{Oka21}. 
We remark that the spherical representations considered in \cite[Theorem 1.5]{Oka21} are those with Arthur parameters that are trivial on the Weil-Deligne group (see \cite[Page 5]{Oka21} for the setting), while general unramified representations have Arthur parameters that are trivial on the subgroup $I_F \times \SL_2(\BC)$, where $I_F$ is the inertia subgroup of the Weil group $W_F$. 

\begin{thm}[Theorem 1.5, \cite{Oka21}]\label{keylemma0}
Let $\sigma_{sn}$ is a strongly negative unramified representation of $G_n$. 
If $G_n=\Sp_{2n}, \RO_{2n}$ and if the Jordan Block of $\sigma_{sn}$ is of the form 
$$
\rm{Jord}(\sigma_{sn})=\{(1_{\GL_1}, 2m_1+1),
\ldots, (1_{\GL_1}, 2m_l+1)\},
$$
then the maximal partitions of the wave-front set of $\sigma_{sn}$ is given by 
$$\frak{p}^m(\sigma_{sn})=\{\eta_{\frak{g}_{n}^{\vee}, \frak{g}_{n}}([(2m_1+1)\cdots(2m_l+1)])\}.$$
If $G_n=\SO_{2n+1}$ and if  the Jordan Block of $\sigma_{sn}$ is of the form 
$$
\rm{Jord}(\sigma_{sn})=\{(1_{\GL_1}, 2m_1),
\ldots, (1_{\GL_1}, 2m_l)\},
$$
then the maximal partitions of the wave-front set of $\sigma_{sn}$ is given by
$$\frak{p}^m(\sigma_{sn})=\{\eta_{\frak{g}_{n}^{\vee}, \frak{g}_{n}}([(2m_1)\cdots(2m_l)])\}.$$
\end{thm}


\section{Arthur parameters and unramified local components}

In this section, in terms of the classification of the unramified unitary dual of 
$G_n$, we study the structure of the unramified local components $\sigma=\pi_v$ of an irreducible automorphic
representation $\pi=\otimes_v\pi_v$ of $G_n(\BA)$ belonging to an automorphic $L^2$-packet $\wt{\Pi}_{\psi}(\varepsilon_{\psi})$
for an arbitrary global Arthur parameter $\psi \in \wt{\Psi}_2(G_n)$.
Then, we prove Theorem \ref{main}. 
We first consider the cases of $G_n=\Sp_{2n}, \SO_{2n+1}, \RO_{2n}^{\alpha}$ and leave the case of $G_n=\RU_n$ to the end of the section. 

\subsection{Unramified structure of Arthur parameters}\label{ssec-USAP}
For a given global Arthur parameter $\psi \in \widetilde{\Psi}_2(G_{n})$, $\wt{\Pi}_{\psi}(\varepsilon_{\psi})$
is the corresponding automorphic $L^2$-packet.
It is clear that the irreducible unramified representations, which are the local components of $\pi\in\wt{\Pi}_{\psi}(\varepsilon_{\psi})$, are determined by
the local Arthur parameter $\psi_v$ at almost all unramified local places $v$ of $k$. 
We fix one of the members, $\pi \in \wt{\Pi}_{\psi}(\varepsilon_{\psi})$,
and describe the unramified local component $\pi_v$ at a finite local place $v$, where the local Arthur parameter
$$
\psi_v=\psi_{1,v}\boxplus\psi_{2,v}\boxplus\cdots\boxplus\psi_{r,v}
$$
is unramified, i.e. $\tau_{i,v}$ for $i=1,2,\cdots,r$ are all unramified, and $G_n(k_v)$ is split. 

We write $F=k_v$ and first consider the case of $G_n=\Sp_{2n}, \RO_{2n}^{\alpha}$.
Rewrite the global Arthur parameter $\psi$ as follows:
\begin{equation}
\psi = [\boxplus_{i=1}^k (\tau_i, 2b_i)] \boxplus [\boxplus_{j=k+1}^{k+l}
(\tau_j, 2b_j+1)] \boxplus [\boxplus_{s=k+l+1}^{k+l+2t+1} (\tau_s, 2b_s+1)],
\end{equation}
where $\tau_i \in \CA_{\cusp}(\GL_{2a_i})$ is of symplectic type for $1 \leq i \leq k$, and $\tau_j \in \CA_{\cusp}(\GL_{2a_j})$ and $\tau_s \in \CA_{\cusp}(\GL_{2a_s+1})$ are of orthogonal type for $k+1 \leq j \leq k+l$ and $k+l+1 \leq s \leq k+l+2t+1$.
Define
\begin{align*}
I&:=\{1,2,\ldots,k\},\\
J&:=\{k+1, k+2, \ldots, k+l\},\\
S&:=\{k+l+1, k+l+2, \ldots, k+l+2t+1\} 
\end{align*} 
Let $J_1$ be the subset of $J$ such that $\omega_{\tau_{j,v}}=1$, and $J_2=J \bs J_1$, that is, for $j \in J_2$, $\omega_{\tau_{j,v}}=\lambda_0$.
Let $S_1$ be the subset of $S$ such that $\omega_{\tau_{s,v}}=1$, and $S_2=S \bs S_1$, that is, for $s \in S_2$, $\omega_{\tau_{s,v}}=\lambda_0$.
From the definition of Arthur parameters, we can easily see that $\#\{J_2\} \cup \#\{S_2\}$ is even, which implies that $\#\{J_2\} \cup \#\{S_1\}$ is odd when $G_n=\Sp_{2n}$, is even when $G_n=\RO_{2n}$. The local unramified Arthur parameter $\psi_v$ has the following structures:
\begin{itemize}
\item For $i \in I$,
$$\tau_{i,v} = \times_{q=1}^{a_i} \nu^{\beta^i_q} \chi^i_q
\times_{q=1}^{a_i} \nu^{-\beta^i_q} \chi_q^{i,-1},$$
where $0 \leq \beta^i_q < \frac{1}{2}$, for $1 \leq q \leq a_i$, and $\chi^i_q$'s are unramified unitary
characters of $F^*$.

\item For $j \in J_1$,
$$\tau_{j,v} = \times_{q=1}^{a_j} \nu^{\beta^j_q} \chi^j_q
\times_{q=1}^{a_j} \nu^{-\beta^j_q} \chi_q^{j,-1},$$
where $0 \leq \beta^j_q < \frac{1}{2}$, for $1 \leq q \leq a_j$, and $\chi^j_q$'s are unramified unitary
characters of $F^*$.

\item For $j \in J_2$,
$$\tau_{j,v} = \times_{q=1}^{a_j-1} \nu^{\beta^j_q} \chi^j_q \times
\lambda_0 \times 1_{GL_1}
\times_{q=1}^{a_j-1} \nu^{-\beta^j_q} \chi_q^{j,-1},$$
where $0 \leq \beta^j_q < \frac{1}{2}$, for $1 \leq q \leq a_j$, and $\chi^j_q$'s are unramified unitary
characters of $F^*$.

\item For $s \in S_1$,
$$\tau_{s,v} = \times_{q=1}^{a_s} \nu^{\beta^s_q} \chi^s_q \times
1_{GL_1}
\times_{q=1}^{a_s} \nu^{-\beta^s_q} \chi_q^{s,-1},$$
where $0 \leq \beta^s_q < \frac{1}{2}$, for $1 \leq q \leq a_s$, and $\chi^s_q$'s are unramified unitary
characters of $F^*$.

\item For $s \in S_2$,
$$\tau_{s,v} = \times_{q=1}^{a_s} \nu^{\beta^s_q} \chi^s_q \times
\lambda_0
\times_{q=1}^{a_s} \nu^{-\beta^s_q} \chi_q^{s,-1},$$
where $0 \leq \beta^s_q < \frac{1}{2}$, for $1 \leq q \leq a_s$, and $\chi^s_q$'s are unramified unitary
characters of $F^*$.
\end{itemize}

We define
\begin{align*}
\rm{Jord}_1
= \ &\{(\lambda_0, 2b_j+1), j \in J_2; (\lambda_0, 2b_s+1), s \in S_2;\\
& (1_{GL_1}, 2b_j+1), j \in J_2; (1_{GL_1}, 2b_s+1), s \in S_1\}.
\end{align*}
Note that $\rm{Jord}_1$ is a multi-set.
Let $\rm{Jord}_2$ be a set consists of different Jordan blocks with odd multiplicities in
$\rm{Jord}_1$. Thus $\rm{Jord}_2$ has the form of \eqref{sec7equ1sp}. By Theorem \ref{thm8}, there is a corresponding irreducible strongly negative unramified representation $\sigma_{sn}$. 
Then we define the following Jordan blocks:
\begin{align*}
\rm{Jord}_I & = \{(\chi_q^i, 2b_i), (\chi_q^{i,-1}, 2b_i), i \in I, 1 \leq q \leq a_i, \beta_q^i=0\},\\
\rm{Jord}_{J_1} & = \{(\chi_q^j, 2b_j+1), (\chi_q^{j,-1}, 2b_j+1), j \in J_1, 1 \leq q \leq a_j, \beta_q^j=0\},\\
\rm{Jord}_{J_2} & = \{(\chi_q^j, 2b_j+1), (\chi_q^{j,-1}, 2b_j+1), j \in J_2, 1 \leq q \leq a_j-1, \beta_q^j=0\},\\
\rm{Jord}_{S_1} & = \{(\chi_q^s, 2b_s+1), (\chi_q^{s,-1}, 2b_s+1), s \in S_1, 1 \leq q \leq a_s, \beta_q^s=0\},\\
\rm{Jord}_{S_2} & = \{(\chi_q^s, 2b_s+1), (\chi_q^{s,-1}, 2b_s+1), s \in S_2, 1 \leq q \leq a_s, \beta_q^s=0\}.
\end{align*}
Finally, we define
$$
\rm{Jord}_3 = (\rm{Jord}_1 \bs \rm{Jord}_2) \cup \rm{Jord}_I \cup \rm{Jord}_{J_1} \cup \rm{Jord}_{J_2} \cup \rm{Jord}_{S_1} \cup \rm{Jord}_{S_2}.
$$ 
By Theorem \ref{thm9}, corresponding to the data $\rm{Jord}_3$ and $\sigma_{sn}$, there is an irreducible negative unramified presentation $\sigma_{neg}$.

Let
\begin{align*}
\textbf{e}_I & = \{(\chi_q^i, 2b_i, \beta_q^i), i \in I, 1 \leq q \leq a_i, \beta_q^i>0\},\\
\textbf{e}_{J_1} & = \{(\chi_q^j, 2b_j+1, \beta_q^j), j \in J_1, 1 \leq q \leq a_j, \beta_q^j>0\},\\
\textbf{e}_{J_2} & = \{(\chi_q^j, 2b_j+1, \beta_q^j), j \in J_2, 1 \leq q \leq a_j-1, \beta_q^j>0\},\\
\textbf{e}_{S_1} & = \{(\chi_q^s, 2b_s+1, \beta_q^s), s \in S_1, 1 \leq q \leq a_s, \beta_q^s>0\},\\
\textbf{e}_{S_2} & = \{(\chi_q^s, 2b_s+1, \beta_q^s), s \in S_2, 1 \leq q \leq a_s, \beta_q^s>0\}.
\end{align*}
Then we define
\begin{equation}\label{sete}
\textbf{e}=\textbf{e}_I \cup \textbf{e}_{J_1}\cup \textbf{e}_{J_2}
\cup \textbf{e}_{S_1} \cup \textbf{e}_{S_2}.
\end{equation}
Since the unramified component $\pi_v$ is unitary, we must have that
$(\textbf{e}, \sigma_{neg}) \in \CM^{u,unr}(n)$, and $\pi_v$ is exactly the irreducible unramified unitary representation $\sigma$ of $G_{n}(F)$ which corresponds to $(\textbf{e}, \sigma_{neg})$ as in Theorem \ref{thm10}.

Now we consider the case of $G_n=\SO_{2n+1}$.
Rewrite the global Arthur parameter $\psi$ as follows:
\begin{equation}
\psi = [\boxplus_{i=1}^k (\tau_i, 2b_i+1)] \boxplus [\boxplus_{j=k+1}^{k+l}
(\tau_j, 2b_j)] \boxplus [\boxplus_{s=k+l+1}^{k+l+2t+1} (\tau_s, 2b_s)],
\end{equation}
where $\tau_i \in \CA_{\cusp}(\GL_{2a_i})$ is of symplectic type for $1 \leq i \leq k$, $\tau_j \in \CA_{\cusp}(\GL_{2a_j})$ and $\tau_s \in \CA_{\cusp}(\GL_{2a_s+1})$ are of orthogonal type for $k+1 \leq j \leq k+l$ and $k+l+1 \leq s \leq k+l+2t+1$.
Similarly, we define
\begin{align*}
    I&:=\{1,2,\ldots,k\},\\
    J&:=\{k+1, k+2, \ldots, k+l\},\\
    S&:=\{k+l+1, k+l+2, \ldots, k+l+2t+1\}
\end{align*}  
Let $J_1$ be the subset of $J$ such that $\omega_{\tau_{j,v}}=1$, and $J_2=J \bs J_1$, that is, for $j \in J_2$, $\omega_{\tau_{j,v}}=\lambda_0$.
Let $S_1$ be the subset of $S$ such that $\omega_{\tau_{s,v}}=1$, and $S_2=S \bs S_1$, that is, for $s \in S_2$, $\omega_{\tau_{s,v}}=\lambda_0$. The local unramified Arthur parameter $\psi_v$ has the following structures:
\begin{itemize}
\item For $i \in I_1$,
$$\tau_{i,v} = \times_{q=1}^{a_i} \nu^{\beta^i_q} \chi^i_q
\times_{q=1}^{a_i} \nu^{-\beta^i_q} \chi_q^{i,-1},$$
where $0 \leq \beta^i_q < \frac{1}{2}$, for $1 \leq q \leq a_i$, and $\chi^i_q$'s are unramified unitary
characters of $F^*$.

\item For $j \in J_1$,
$$\tau_{j,v} = \times_{q=1}^{a_j} \nu^{\beta^j_q} \chi^j_q
\times_{q=1}^{a_j} \nu^{-\beta^j_q} \chi_q^{j,-1},$$
where $0 \leq \beta^j_q < \frac{1}{2}$, for $1 \leq q \leq a_j$, and $\chi^j_q$'s are unramified unitary
characters of $F^*$.

\item For $j \in J_2$,
$$\tau_{j,v} = \times_{q=1}^{a_j-1} \nu^{\beta^j_q} \chi^j_q \times
\lambda_0 \times 1_{GL_1}
\times_{q=1}^{a_j-1} \nu^{-\beta^j_q} \chi_q^{j,-1},$$
where $0 \leq \beta^j_q < \frac{1}{2}$, for $1 \leq q \leq a_j$, and $\chi^j_q$'s are unramified unitary
characters of $F^*$.

\item For $s \in S_1$,
$$\tau_{s,v} = \times_{q=1}^{a_s} \nu^{\beta^s_q} \chi^s_q \times
1_{GL_1}
\times_{q=1}^{a_s} \nu^{-\beta^s_q} \chi_q^{s,-1},$$
where $0 \leq \beta^s_q < \frac{1}{2}$, for $1 \leq q \leq a_s$, and $\chi^s_q$'s are unramified unitary
characters of $F^*$.

\item For $s \in S_2$,
$$\tau_{s,v} = \times_{q=1}^{a_s} \nu^{\beta^s_q} \chi^s_q \times
\lambda_0
\times_{q=1}^{a_s} \nu^{-\beta^s_q} \chi_q^{s,-1},$$
where $0 \leq \beta^s_q < \frac{1}{2}$, for $1 \leq q \leq a_s$, and $\chi^s_q$'s are unramified unitary
characters of $F^*$.
\end{itemize}

We define
\begin{align*}
\rm{Jord}_1
= \ &\{(\lambda_0, 2b_j), j \in J_2; (\lambda_0, 2b_s), s \in S_2;\\
& (1_{GL_1}, 2b_j), j \in J_2; (1_{GL_1}, 2b_s), s \in S_1\}.
\end{align*}
Note that $\rm{Jord}_1$ is a multi-set.
Let $\rm{Jord}_2$ be a set consists of different Jordan blocks with odd multiplicities in
$\rm{Jord}_1$. Thus $\rm{Jord}_2$ has the form of \eqref{sec7equ1sp}. By Theorem \ref{thm8}, there is a corresponding irreducible strongly negative unramified representation $\sigma_{sn}$. 
Then we define the following Jordan blocks:
\begin{align*}
\rm{Jord}_I & = \{(\chi_q^i, 2b_i+1), (\chi_q^{i,-1}, 2b_i+1), i \in I, 1 \leq q \leq a_i, \beta_q^i=0\},\\
\rm{Jord}_{J_1} & = \{(\chi_q^j, 2b_j), (\chi_q^{j,-1}, 2b_j), j \in J_1, 1 \leq q \leq a_j, \beta_q^j=0\},\\
\rm{Jord}_{J_2} & = \{(\chi_q^j, 2b_j), (\chi_q^{j,-1}, 2b_j), j \in J_2, 1 \leq q \leq a_j-1, \beta_q^j=0\},\\
\rm{Jord}_{S_1} & = \{(\chi_q^s, 2b_s), (\chi_q^{s,-1}, 2b_s), s \in S_1, 1 \leq q \leq a_s, \beta_q^s=0\},\\
\rm{Jord}_{S_2} & = \{(\chi_q^s, 2b_s), (\chi_q^{s,-1}, 2b_s), s \in S_2, 1 \leq q \leq a_s, \beta_q^s=0\}.
\end{align*}
Finally, we define
$$\rm{Jord}_3 = (\rm{Jord}_1 \bs \rm{Jord}_2) \cup \rm{Jord}_I \cup \rm{Jord}_{J_1} \cup \rm{Jord}_{J_2} \cup \rm{Jord}_{S_1} \cup \rm{Jord}_{S_2}.$$ By Theorem \ref{thm9}, corresponding to the data $\rm{Jord}_3$ and $\sigma_{sn}$, there is an irreducible negative unramified presentation $\sigma_{neg}$.

Let
\begin{align*}
\textbf{e}_I & = \{(\chi_q^i, 2b_i+1, \beta_q^i), i \in I, 1 \leq q \leq a_i, \beta_q^i>0\},\\
\textbf{e}_{J_1} & = \{(\chi_q^j, 2b_j, \beta_q^j), j \in J_1, 1 \leq q \leq a_j, \beta_q^j>0\},\\
\textbf{e}_{J_2} & = \{(\chi_q^j, 2b_j, \beta_q^j), j \in J_2, 1 \leq q \leq a_j-1, \beta_q^j>0\},\\
\textbf{e}_{S_1} & = \{(\chi_q^s, 2b_s, \beta_q^s), s \in S_1, 1 \leq q \leq a_s, \beta_q^s>0\},\\
\textbf{e}_{S_2} & = \{(\chi_q^s, 2b_s, \beta_q^s), s \in S_2, 1 \leq q \leq a_s, \beta_q^s>0\}.
\end{align*}
Then we define
\begin{equation}\label{sete2}
\textbf{e}=\textbf{e}_I \cup \textbf{e}_{J_1}\cup \textbf{e}_{J_2}
\cup \textbf{e}_{S_1} \cup \textbf{e}_{S_2}.
\end{equation}
Since the unramified component $\pi_v$ is unitary, we must have that
$(\textbf{e}, \sigma_{neg}) \in \CM^{u,unr}(n)$, and $\pi_v$ is exactly the irreducible unramified unitary representation $\sigma$ of $G_{n}(F)$ which corresponds to $(\textbf{e}, \sigma_{neg})$ as in Theorem \ref{thm10}.

\subsection{Proof of Theorem \ref{main}}\label{ssec-PThM}
The following result from \cite{JL16} is needed for the proof of Theorem \ref{main}. 

\begin{prop}[Proposition 6.1, \cite{JL16}]\label{propsq}
For any finitely many non-square elements $\alpha_i \notin k^*/(k^*)^2, 1 \leq i \leq t$,  there are infinitely many finite places $v$ such that $\alpha_i \in (k_v^*)^2$, for any $1 \leq i \leq t$.
\end{prop}

Now we are going to prove Theorem \ref{main}. 
First we consider the cases of $G_n=\Sp_{2n}$, $\SO_{2n+1}$, $\RO_{2n}^{\alpha}$. 
Given any $\psi = \boxplus_{i=1}^r (\tau_i, b_i) \in \wt{\Psi}_2(G_{n})$. Assume that $\{\tau_{i_1}, \ldots, \tau_{i_q}\}$ is a multi-set of all the $\tau$'s with non-trivial central characters.
Since all $\tau_{i_j}$'s are self-dual, the central characters $\omega_{\tau_{i_j}}$'s are all quadratic characters, which are parametrized by global non-square elements. Assume that $\omega_{\tau_{i_j}}=\chi_{\alpha_{i_j}}$,
where $\alpha_{i_j} \in k^* / (k^*)^2$,
and $\chi_{\alpha_{i_j}}$ is the quadratic character given by the global Hilbert symbol $(\cdot, \alpha_{i_j})$.
Note that $\{\alpha_{i_1}, \ldots, \alpha_{i_q}\}$ is a multi-set.
By Proposition \ref{propsq}, there are infinitely many finite places $v$, such that $\alpha$ and $\alpha_{i_j}$'s are all squares in
$k_v$.
Therefore, for the given $\psi$,
there are infinitely many finite places $v$ such that $G_{n}(k_v)$ split and all $\tau_{i,v}$'s have trivial central characters. From
the discussion in Section 3, for any $\pi \in \wt{\Pi}_{\psi}(\epsilon_{\psi})$, there is a finite local place $v$ with such a property
that $\pi_v$ is an irreducible unramified unitary representation of \textbf{Type (I)} as in \eqref{typeA}.

We are going to discuss the connection with the classification of D. Barbasch in \cite{Bar10}. 

Assume first that $G_n=\Sp_{2n}, \RO_{2n}^{\alpha}$. If $\sigma$ is an irreducible unramified unitary representation of $G_{n}(k_v)$ corresponding to the pair $(\textbf{e}, \sigma_{neg}) \in \CM^{u,unr}(n)$, then the orbit $\check{\CO}$ corresponding to $\sigma$ in \cite{Bar10} is given by the following partition:
\begin{align}
[(\prod_{j=1}^t n_j^2) (\prod_{(\chi,m,\alpha) \in \textbf{e}} m^2) (\prod_{i=1}^k (2n_i+1))(\prod_{i=1}^l (2m_i+1))].
\end{align}
When $\pi_v$ is of {\bf Type (I)} as in \eqref{typeA}, the orbit $\check{\CO}$ corresponding to $\sigma=\pi_v$ in \cite{Bar10} is given by the following partition:
\begin{align}\label{partition-1}
    [(\prod_{j=1}^t n_j^2) (\prod_{(\chi,m,\alpha) \in \textbf{e}} m^2) (\prod_{i=1}^l (2m_i+1))]
\end{align}
which turns out to be $\ul{p}(\psi)$ exactly.

Assume now that $G_n=\SO_{2n+1}$. If $\sigma$ is an irreducible unramified unitary representation of $G_{n}(k_v)$ corresponding to the pair $(\textbf{e}, \sigma_{neg}) \in \CM^{u,unr}(n)$, then the orbit $\check{\CO}$ corresponding to $\sigma$ in \cite{Bar10} is given by the following partition:
\begin{align}\label{partition-2}
    [(\prod_{j=1}^t n_j^2) (\prod_{(\chi,m,\alpha) \in \textbf{e}} m^2) (\prod_{i=1}^k (2n_i))(\prod_{i=1}^l (2m_i))].
\end{align}
When $\pi_v$ is of {\bf Type I} as in \eqref{typeA}, the orbit $\check{\CO}$ corresponding to $\sigma=\pi_v$ in \cite{Bar10} is given by the following partition:
\begin{align}
    [(\prod_{j=1}^t n_j^2) (\prod_{(\chi,m,\alpha) \in \textbf{e}} m^2) (\prod_{i=1}^l (2m_i))],
\end{align}
which turns out to be $\ul{p}(\psi)$ exactly.

We claim that for the cases of $G_n=\Sp_{2n}, \SO_{2n+1}, \RO_{2n}^{\alpha}$, 
Theorem \ref{main} can be deduced from the following theorem whose proof will be given in the next three sections. 

\begin{thm}\label{main2}
Let $\sigma$ be an irreducible unramified unitary representations of $G_n(k_v)$ 
of {\bf Type (I)} as in \eqref{typeA}. For any $\ul{p} \in \frak{p}^m(\sigma)$, 
 the following bound 
\begin{align}\label{partition-3}
   \ul{p} \leq \eta_{\frak{g}_{n}^\vee, \frak{g}_{n}} [(\prod_{j=1}^t n_j^2) (\prod_{(\chi,m,\alpha) \in \textbf{e}} m^2) (\prod_{i=1}^l (2m_i+1))]
   \end{align}
   holds with the partition on the left-hand side from \eqref{partition-1} when $G_n=\Sp_{2n}, \RO_{2n}$; and 
   the following bound
   \begin{align}\label{partition-4}
      \ul{p} \leq \eta_{\frak{g}_{n}^\vee, \frak{g}_{n}} [(\prod_{j=1}^t n_j^2) (\prod_{(\chi,m,\alpha) \in \textbf{e}} m^2) (\prod_{i=1}^l (2m_i))]
\end{align}
holds with the partition on the left-hand side from \eqref{partition-2} when $G_n=\SO_{2n+1}$.
\end{thm}

For the case of $G_n=\RU_n$, $E/k$ a quadratic extension, by similar arguments, for any $\pi \in \wt{\Pi}_{\psi}(\epsilon_{\psi})$, there is a finite local place $v$ such that $G_n(E_v)=\GL_n(k_v) \times \GL_n(k_v)$, split, and $\pi_v$ is unramified. Then, Theorem \ref{main} is simply implied by the classfication of the unramified unitary dual of $\GL_n$ (\cite{Tad86}) and the result of Moeglin and Waldspurger on the wave front set of representations of $\GL_n$ (\cite[Section II.2]{MW87}). Note that for $G_n=\RU_n$, the Barbasch-Vogan-Spaltenstein duality is just the transpose of partitions. We omit the details here. 

This completes the proof of Theorem \ref{main}. 

\begin{rmk}
We expect that the method of proving Theorem \ref{main} in this paper also applies to the inner forms of even orthogonal groups, once the full Arthur classification of the discrete spectrum being carried out (see \cite{CZ21a, CZ21b} for recent progress in this direction). The same method can also be applied to the metaplectic double cover of symplectic groups, whose proof will appear elsewhere. 
Note that for the metaplectic double cover of symplectic groups, the notion of Barbasch-Vogan-Spaltenstein duality has been defined in \cite{BMSZ20}.
\end{rmk}


\section{Proof of Theorem \ref{main2}, $G_n=\Sp_{2n}$}

First, we recall the following general lemma which can be deduced from the argument in
\cite[Section II.1.3]{MW87}. 

\begin{lem}[Section II.1.3, \cite{MW87}]\label{keylemma}
Let $G$ be a reductive group defined over a non-Archimedean local field $F$, and $Q=MN$ be a parabolic subgroup of $G$.
Let $\delta$ be an irreducible admissible representation of $M$. Then 
$$\frak{n}^m(\Ind_Q^G \delta)=\{\Ind_{\frak{q}}^{\frak{g}} \mathcal{O} : \mathcal{O} \in \frak{n}^m(\delta)\},$$
where $\frak{q}$ and $\frak{g}$ are the Lie algebras of $Q$ and $G$, respectively. For induced nilpotent orbits, see \cite[Chapter 7]{CM93}. 
\end{lem}

Now we prove Theorem \ref{main2} for the case that $G_n=\Sp_{2n}$.
By the assumption of Theorem \ref{main2}, $\sigma$ is of {\bf Type (I)} and is of the following form:
\begin{equation*}
\sigma = \times_{(\chi, m, \alpha) \in \textbf{e}}
v^{\alpha} \chi({\det}_m) \rtimes \sigma_{neg},
\end{equation*}
where $\sigma_{neg}$ is the unique irreducible negative unramified  subrepresentation
of the following induced representation
\begin{equation*}
\chi_1 ({\det}_{{n_1}}) \times \cdots \times \chi_t ({\det}_{{n_t}}) \rtimes \sigma_{sn},
\end{equation*}
with $\sigma_{sn}$ being the unique strongly negative unramified constituent of the following induced representation:
\begin{align*}
\begin{split}
& \nu^{\frac{m_{l-1}-m_l}{2}} 1_{{\det}_{m_{l-1}+m_l+1}}
\times
\nu^{\frac{m_{l-3}-m_{l-2}}{2}} 1_{{\det}_{m_{l-3}+m_{l-2}+1}}\\
& \qquad\times \cdots \times
\nu^{\frac{m_{2}-m_3}{2}} 1_{{\det}_{m_{2}+m_3+1}}
\rtimes 1_{\Sp_{2m_1}}.
\end{split}
\end{align*}
As representations of general linear groups, it is known that 
$$\frak{p}^m(\chi({\det}_k))=\{[1^k]\},$$
for any given character $\chi$ and any integer $k$. By Lemma \ref{keylemma}, we have 
\begin{align*}
    &\frak{p}^m(\times_{(\chi, m, \alpha) \in \textbf{e}}
v^{\alpha} \chi({\det}_m) \times \chi_1 ({\det}_{{n_1}}) \times \cdots \times \chi_t ({\det}_{{n_t}}))\\
&\quad=\{+_{(\chi, m, \alpha) \in \textbf{e}} [1^m] + [1^{n_1}]+ \cdots [1^{n_t}]\}=\{[(\prod_{(\chi, m, \alpha) \in \textbf{e}} m) (\prod_{i=1}^t n_i)]^t\}.
\end{align*}
By Theorem \ref{keylemma0}, Lemma \ref{keylemma}, and by \cite[Theorem 7.3.3]{CM93} on formula for induced nilpotent orbits, for any $\ul{p} \in \frak{p}^m(\sigma)$, we have 
$$\ul{p} \leq (2[(\prod_{(\chi, m, \alpha) \in \textbf{e}} m) (\prod_{i=1}^t n_i)]^t + \eta_{\frak{so}_{2k+1}, \frak{sp}_{2k}}([\prod_{j=1}^l(2m_j+1)]))_{\Sp_{2n}},$$
where $2k=(\sum_{i=1}^l (2m_i+1)) -1$. 

To prove Theorem \ref{main2} in the case, it suffices to show the following lemma.

\begin{lem}\label{keylemma2}
The following identity
\[
\eta_{\frak{so}_{2n+1}, \frak{sp}_{2n}} (\ul{p}(\psi)) = \left(2[(\prod_{(\chi, m, \alpha) \in \textbf{e}} m) (\prod_{i=1}^t n_i)]^t + \eta_{\frak{so}_{2k+1}, \frak{sp}_{2k}}([\prod_{j=1}^l(2m_j+1)])\right)_{\Sp_{2n}}
\]
holds with $2k=(\sum_{i=1}^l (2m_i+1)) -1$.
\end{lem}

\begin{proof}
Recall that
$$\ul{p}(\psi)=[(\prod_{(\chi, m, \alpha) \in \textbf{e}} m^2) (\prod_{i=1}^t n_i^2)(\prod_{j=1}^l(2m_j+1))]$$
and 
$$\eta_{\frak{so}_{2n+1}, \frak{sp}_{2n}} (\ul{p}(\psi)) = ((\ul{p}(\psi)^-)_{\Sp_{2n}})^t,$$
where given any partition $\ul{p}=[p_r \cdots p_1]$ with $p_r \geq \cdots \geq p_1$,  we have that $\ul{p}^-=[p_r \cdots (p_1-1)]$.
On the other hand, we have 
\begin{align*}
    &\left(2[(\prod_{(\chi, m, \alpha) \in \textbf{e}} m) (\prod_{i=1}^t n_i)]^t + \eta_{\frak{so}_{2k+1}, \frak{sp}_{2k}}([\prod_{j=1}^l(2m_j+1)])\right)_{\Sp_{2n}}\\
    &\quad=\left(2[(\prod_{(\chi, m, \alpha) \in \textbf{e}} m) (\prod_{i=1}^t n_i)]^t + (([\prod_{j=1}^l(2m_j+1)]^-)_{Sp_{2k}})^t\right)_{\Sp_{2n}}\\
    &\quad\quad=\left([(\prod_{(\chi, m, \alpha) \in \textbf{e}} m^2) (\prod_{i=1}^t n_i^2)(([\prod_{j=1}^l(2m_j+1)]^-)_{Sp_{2k}})]^t\right)_{\Sp_{2n}}.
\end{align*}
Given any partition $\ul{p}$ of $\Sp_{2n}$, it is known that $(\ul{p}^{\Sp_{2n}})^t=(\ul{p}^t)_{\Sp_{2n}}$ (see \cite[Proof of Theorem 6.3.11]{CM93}). Note that 
$$[(\prod_{(\chi, m, \alpha) \in \textbf{e}} m^2) (\prod_{i=1}^t n_i^2)(([\prod_{j=1}^l(2m_j+1)]^-)_{\Sp_{2k}})]$$
is indeed a symplectic partition. 
Hence we have 
\begin{align*}
&\left([(\prod_{(\chi, m, \alpha) \in \textbf{e}} m^2) (\prod_{i=1}^t n_i^2)(([\prod_{j=1}^l(2m_j+1)]^-)_{\Sp_{2k}})]^t\right)_{\Sp_{2n}}\\
&\qquad=\left([(\prod_{(\chi, m, \alpha) \in \textbf{e}} m^2) (\prod_{i=1}^t n_i^2)(([\prod_{j=1}^l(2m_j+1)]^-)_{\Sp_{2k}})]^{\Sp_{2n}}\right)^t.
\end{align*}
Therefore, we only need to show that
\begin{equation}\label{keylemma2equ1}
   (\ul{p}(\psi)^-)_{\Sp_{2n}} = [(\prod_{(\chi, m, \alpha) \in \textbf{e}} m^2) (\prod_{i=1}^t n_i^2)(([\prod_{j=1}^l(2m_j+1)]^-)_{\Sp_{2k}})]^{\Sp_{2n}}. 
\end{equation}
Note that 
$$
\left([\prod_{j=1}^l(2m_j+1)]^-\right)_{\Sp_{2k}} 
= [(2m_l)(2m_{l-1}+2) \cdots (2m_3)(2m_2+2)(2m_1)].$$

We have to rewrite the partition 
\[
[(\prod_{(\chi, m, \alpha) \in \textbf{e}} m^2) (\prod_{i=1}^t n_i^2)]
\]
as
$[k_s^2k_{s-1}^2\cdots k_1^2]$ with $k_s \geq k_{s-1} \geq \cdots \geq k_1$. 
To proceed, we separate into the following cases: 
\begin{enumerate}
    \item $k_1 \geq 2m_1+1$;
    \item $k_1 < 2m_1+1$.
\end{enumerate}
In each case, for $1 \leq j \leq \frac{l-1}{2}$, we list all the different odd $k_i$'s between $2m_{2j+1}+1$ and $2m_{2j}+1$ as 
$$2m_{2j+1}+1 > k_j^1 > k_j^2 > \cdots > k_j^{s_j} > 2m_{2j}+1.$$

\quad

{\bf Case (1):} $k_1 \geq 2m_1+1$. We have
$$\ul{p}(\psi)^-=[(k_s^2k_{s-1}^2\cdots k_1^2)(\prod_{j=2}^l(2m_j+1))(2m_1)].$$
Then $(\ul{p}(\psi)^-)_{Sp_{2n}}$ is obtained from $\ul{p}(\psi)^-$ via replacing 
\[
(2m_{2j+1}+1, 2m_{2j}+1)
\]
by $(2m_{2j+1}, 2m_{2j}+2)$, and $k_j^{i,2}$ by $(k_j^i+1,k_j^i-1)$, for $1 \leq j \leq \frac{l-1}{2}$, $1 \leq i \leq s_j$. On the other hand, we have 
\begin{align*}
    &[k_s^2\cdots k_1^2(([\prod_{j=1}^l(2m_j+1)]^-)_{\Sp_{2k}})]^{\Sp_{2n}}\\
    &\qquad=[k_s^2\cdots k_1^2(2m_l)(2m_{l-1}+2) \cdots (2m_3)(2m_2+2)(2m_1)]^{\Sp_{2n}},
\end{align*}
which is obtained from 
\[
[k_s^2\cdots k_1^2(2m_l)(2m_{l-1}+2) \cdots (2m_3)(2m_2+2)(2m_1)]
\]
via 
replacing $k_j^{i,2}$ by $(k_j^i+1,k_j^i-1)$, for $1 \leq j \leq \frac{l-1}{2}$, $1 \leq i \leq s_j$.
Hence, we deduce that \eqref{keylemma2equ1} holds. 

\quad

{\bf Case (2):} $k_1 < 2m_1+1$. We have
$$\ul{p}(\psi)^-=[(k_s^2k_{s-1}^2\cdots k_2^2)(\prod_{j=1}^l(2m_j+1))(k_1)(k_1-1)].$$
To carry out the $\Sp_{2n}$-collapse of
$\ul{p}(\psi)^-$, we also need to list all the different odd $k_i$'s between $2m_1+1$ and $k_1$ as
$$
2m_{1}+1 > k_0^1 > k_0^2 > \cdots > k_0^{s_0} > k_1.
$$
Then $(\ul{p}(\psi)^-)_{\Sp_{2n}}$ is obtained from $\ul{p}(\psi)^-$ via  replacing 
\[
(2m_{2j+1}+1, 2m_{2j}+1)
\]
by $(2m_{2j+1}, 2m_{2j}+2)$ and $k_j^{i,2}$ by $(k_j^i+1,k_j^i-1)$, for $1 \leq j \leq \frac{l-1}{2}$, $1 \leq i \leq s_j$; and then replacing $(2m_{1}+1, k_1-1)$ by $(2m_{1}, k_1)$ if $k_1$ is even, $(2m_{1}+1, k_1)$ by $(2m_{1}, k_1+1)$ if $k_1$ is odd, and $k_0^{i,2}$ by $(k_0^i+1,k_0^i-1)$, for $1 \leq i \leq s_0$. 
On the other hand, we obtain 
\begin{align*}
    &[k_s^2\cdots k_1^2(([\prod_{j=1}^l(2m_j+1)]^-)_{\Sp_{2k}})]^{\Sp_{2n}}\\
    &\qquad=[k_s^2\cdots k_1^2(2m_l)(2m_{l-1}+2) \cdots (2m_3)(2m_2+2)(2m_1)]^{\Sp_{2n}},
\end{align*}
which is obtained from 
\[
[k_s^2\cdots k_1^2(2m_l)(2m_{l-1}+2) \cdots (2m_3)(2m_2+2)(2m_1)]
\]
via 
replacing $k_j^{i,2}$ by $(k_j^i+1,k_j^i-1)$, for $1 \leq j \leq \frac{l-1}{2}$, $1 \leq i \leq s_j$;
and then replacing $k_1^2$ by $(k_1+1,k_1-1)$ if $k_1$ is odd, 
$k_0^{i,2}$ by $(k_0^i+1,k_0^i-1)$, for $1 \leq i \leq s_0$. 
Hence, we deduce that \eqref{keylemma2equ1} still holds.

This completes the proof of the lemma. 
\end{proof}

The proof of Theorem \ref{main2} has been completed for $G_n=\Sp_{2n}$.

\section{Proof of Theorem \ref{main2}, $G_n=\SO_{2n+1}$}

By the assumption of Theorem \ref{main2}, $\sigma$ is of 
{\bf Type (I)} and is of the following form:
\begin{equation*}
\sigma = \times_{(\chi, m, \alpha) \in \textbf{e}}
v^{\alpha} \chi({\det}_m) \rtimes \sigma_{neg},
\end{equation*}
where $\sigma_{neg}$ is the unique irreducible negative unramified  subrepresentation
of the following induced representation
\begin{equation*}
\chi_1 ({\det}_{{n_1}}) \times \cdots \times \chi_t ({\det}_{{n_t}}) \rtimes \sigma_{sn},
\end{equation*}
with $\sigma_{sn}$ being the unique strongly negative unramified constituent of the following induced representation:
\begin{align*}
\begin{split}
&\nu^{\frac{m_{l-1}-m_l}{2}} 1_{{\det}_{m_{l-1}+m_l}}
\times
\nu^{\frac{m_{l-3}-m_{l-2}}{2}} 1_{{\det}_{m_{l-3}+m_{l-2}}}\\
& \qquad\times \cdots \times
\nu^{\frac{m_{1}-m_2}{2}} 1_{{\det}_{m_{1}+m_2}}
\rtimes 1_{\SO_1}.
\end{split}
\end{align*}
Also recall that $l$ is even, and $0 \leq m_1 < m_2 < \cdots < m_l$. 

As in Section 5, by Lemma \ref{keylemma}, we have 
\begin{align*}
    &\frak{p}^m(\times_{(\chi, m, \alpha) \in \textbf{e}}
v^{\alpha} \chi({\det}_m) \times \chi_1 ({\det}_{{n_1}}) \times \cdots \times \chi_t ({\det}_{{n_t}}))\\
&\qquad=\{[(\prod_{(\chi, m, \alpha) \in \textbf{e}} m) (\prod_{i=1}^t n_i)]^t\}.
\end{align*}
By Theorem \ref{keylemma0}, Lemma \ref{keylemma}, and by \cite[Theorem 7.3.3]{CM93} on formula for induced nilpotent orbits, any $\ul{p} \in \frak{p}^m(\sigma)$ has the following upper bound
$$\ul{p} \leq \left(2[(\prod_{(\chi, m, \alpha) \in \textbf{e}} m) (\prod_{i=1}^t n_i)]^t + \eta_{ \frak{sp}_{2k},\frak{so}_{2k+1}}([\prod_{j=1}^l(2m_j)])\right)_{\SO_{2n+1}},$$
where $2k=\sum_{i=1}^l (2m_i)$. 

To prove Theorem \ref{main2} in this case, it suffices to show the following lemma.

\begin{lem}\label{keylemma3}
The following identity 
$$\eta_{\frak{sp}_{2n},\frak{so}_{2n+1}} (\ul{p}(\psi)) = \left(2[(\prod_{(\chi, m, \alpha) \in \textbf{e}} m) (\prod_{i=1}^t n_i)]^t + \eta_{ \frak{sp}_{2k},\frak{so}_{2k+1}}([\prod_{j=1}^l(2m_j)])\right)_{\SO_{2n+1}}$$
holds with $2k=\sum_{i=1}^l (2m_i)$.
\end{lem}

\begin{proof}
Recall that
$$\ul{p}(\psi)=[(\prod_{(\chi, m, \alpha) \in \textbf{e}} m^2) (\prod_{i=1}^t n_i^2)(\prod_{j=1}^l(2m_j))]$$
and 
$$\eta_{ \frak{sp}_{2k},\frak{so}_{2k+1}} (\ul{p}(\psi)) = ((\ul{p}(\psi)^+)_{SO_{2n+1}})^t,$$
where for any given partition $\ul{p}=[p_r \cdots p_1]$ with $p_r \geq \cdots \geq p_1$, we have $\ul{p}^+=[(p_r+1) \cdots p_1]$.
On the other hand, we have 
\begin{align*}
    &\left(2[(\prod_{(\chi, m, \alpha) \in \textbf{e}} m) (\prod_{i=1}^t n_i)]^t + \eta_{ \frak{sp}_{2k},\frak{so}_{2k+1}}([\prod_{j=1}^l(2m_j)])\right)_{\SO_{2n+1}}\\
    &\quad=\left(2[(\prod_{(\chi, m, \alpha) \in \textbf{e}} m) (\prod_{i=1}^t n_i)]^t + (([\prod_{j=1}^l(2m_j)]^+)_{\SO_{2k+1}})^t\right)_{\SO_{2n+1}}\\
    &\quad\quad=\left([(\prod_{(\chi, m, \alpha) \in \textbf{e}} m^2) (\prod_{i=1}^t n_i^2)(([\prod_{j=1}^l(2m_j)]^+)_{\SO_{2k+1}})]^t\right)_{\SO_{2n+1}}.
\end{align*}
Given any partition $\ul{p}$ of $\SO_{2n+1}$, it is known that $(\ul{p}^{\SO_{2n+1}})^t=(\ul{p}^t)_{\SO_{2n+1}}$ (see \cite[Proof of Theorem 6.3.11]{CM93}). Note that 
$$[(\prod_{(\chi, m, \alpha) \in \textbf{e}} m^2) (\prod_{i=1}^t n_i^2)(([\prod_{j=1}^l(2m_j)]^+)_{\SO_{2k+1}})]$$
is indeed an orthogonal partition. 
Hence we obtain that 
\begin{align*}
&\left([(\prod_{(\chi, m, \alpha) \in \textbf{e}} m^2) (\prod_{i=1}^t n_i^2)(([\prod_{j=1}^l(2m_j)]^+)_{\SO_{2k+1}})]^t\right)_{\SO_{2n+1}}\\
&\qquad=\left([(\prod_{(\chi, m, \alpha) \in \textbf{e}} m^2) (\prod_{i=1}^t n_i^2)(([\prod_{j=1}^l(2m_j)]^+)_{\SO_{2k+1}})]^{\SO_{2n+1}}\right)^t.
\end{align*}
Therefore, we only need to show that
\begin{equation}\label{keylemma3equ1}
   (\ul{p}(\psi)^+)_{\SO_{2n+1}} = [(\prod_{(\chi, m, \alpha) \in \textbf{e}} m^2) (\prod_{i=1}^t n_i^2)(([\prod_{j=1}^l(2m_j)]^+)_{\SO_{2k+1}})]^{\SO_{2n+1}}. 
\end{equation}
Note that the partition $([\prod_{j=1}^l(2m_j)]^+)_{\SO_{2k+1}}$ 
is equal to 
\begin{align*}
   [(2m_l+1)(2m_{l-1}-1)(2m_{l-2}+1) \cdots (2m_3-1)(2m_2+1)(2m_1-1)1], 
\end{align*}
where we omit the ``$(2m_1-1)1$"-term if $2m_1=0$. 

We are going to rewrite the partition $[(\prod_{(\chi, m, \alpha) \in \textbf{e}} m^2) (\prod_{i=1}^t n_i^2)]$ as
$[k_s^2k_{s-1}^2\cdots k_1^2]$ with $k_s \geq k_{s-1} \geq \cdots \geq k_1$. 
To proceed, we separate into the following cases: 
\begin{enumerate}
    \item $k_s \leq 2m_l$;
    \item $k_s > 2m_l$.
\end{enumerate}
In each case, for $1 \leq j \leq \frac{l-2}{2}$, we list all the different even $k_i$'s between $2m_{2j+1}$ and $2m_{2j}$ as 
$$2m_{2j+1} > k_j^1 > k_j^2 > \cdots > k_j^{s_j} > 2m_{2j}.$$

\quad

{\bf Case (1):} $k_s \leq 2m_l$. We have
$$\ul{p}(\psi)^+=[(k_s^2\cdots k_1^2)(2m_l+1)\prod_{j=1}^{l-1}(2m_j)].$$
If $2m_1\neq 0$, to carry out the $\SO_{2n+1}$-collapse of
$\ul{p}(\psi)^+$, 
we also need to list all the different even $k_i$'s between $2m_{1}$ and $0$ as 
$$2m_{1} > k_0^1 > k_0^2 > \cdots > k_0^{s_0} > 0.$$
Then $(\ul{p}(\psi)^+)_{\SO_{2n+1}}$ is obtained from $\ul{p}(\psi)^+$ via replacing $(2m_{2j+1}, 2m_{2j})$ by $(2m_{2j+1}-1, 2m_{2j}+1)$ and $k_j^{i,2}$ by $(k_j^i+1,k_j^i-1)$, for $1 \leq j \leq \frac{l-2}{2}$ and $1 \leq i \leq s_j$; and then 
replacing $(2m_{1}, 0)$ by $(2m_{1}-1, 1)$ and $k_0^{i,2}$ by $(k_0^i+1,k_0^i-1)$ with $1 \leq i \leq s_0$, if $2m_1\neq 0$.
On the other hand, we have 
\begin{align*}
    &[k_s^2\cdots k_1^2(([\prod_{j=1}^l(2m_j)]^+)_{\SO_{2k+1}})]^{\SO_{2n+1}}\\
    &\quad=[k_s^2\cdots k_1^2(2m_l+1)(2m_{l-1}-1)(2m_{l-2}+1) \cdots (2m_1-1)1]^{\SO_{2n+1}},
\end{align*}
which is obtained from 
$$[k_s^2\cdots k_1^2(2m_l+1)(2m_{l-1}-1)(2m_{l-2}+1) \cdots (2m_1-1)1]$$ via 
replacing $k_j^{i,2}$ by $(k_j^i+1,k_j^i-1)$ for $1 \leq j \leq \frac{l-2}{2}$, $1 \leq i \leq s_j$; and then replacing 
$k_0^{i,2}$ by $(k_0^i+1,k_0^i-1)$, $1 \leq i \leq s_0$.
Hence \eqref{keylemma3equ1} holds in this case.

\quad

{\bf Case (2):} $k_s > 2m_l$. We have
$$\ul{p}(\psi)^+=[((k_s+1)k_sk_{s-1}^2\cdots k_1^2)(\prod_{j=1}^l(2m_j))].$$
To carry out the $\SO_{2n+1}$-collapse of
$\ul{p}(\psi)^+$, we also need to list 
all the different even $k_i$'s between $k_s$ and $2m_{l}$ as 
$$k_s > k_l^1 > k_l^2 > \cdots > k_l^{s_l} > 2m_{l},$$ and if $2m_1\neq 0$, list all the different even $k_i$'s between $2m_{1}$ and $0$ as 
$$2m_{1} > k_0^1 > k_0^2 > \cdots > k_0^{s_0} > 0.$$
Then $(\ul{p}(\psi)^+)_{\SO_{2n+1}}$ is obtained from $\ul{p}(\psi)^+$ via replacing $(2m_{2j+1}, 2m_{2j})$ by $(2m_{2j+1}-1, 2m_{2j}+1)$ and $k_j^{i,2}$ by $(k_j^i+1,k_j^i-1)$ for $1 \leq j \leq \frac{l-2}{2}$ and $1 \leq i \leq s_j$; and  
replacing $(k_s+1, 2m_l)$ by $(k_s, 2m_l+1)$ if $k_s$ is odd and $(k_s, 2m_l)$ by $(k_s-1, 2m_l+1)$ if $k_s$ is even, and $k_l^{i,2}$ by $(k_l^i+1,k_l^i-1)$ with $1 \leq i \leq s_l$; 
and finally 
replacing $(2m_{1}, 0)$ by $(2m_{1}-1, 1)$ and $k_0^{i,2}$ by $(k_0^i+1,k_0^i-1)$ for $1 \leq i \leq s_0$, if $2m_1\neq 0$.
On the other hand, we have 
\begin{align*}
    &[k_s^2\cdots k_1^2(([\prod_{j=1}^l(2m_j)]^+)_{\SO_{2k+1}})]^{\SO_{2n+1}}\\
    &\quad=[k_s^2\cdots k_1^2(2m_l+1)(2m_{l-1}-1)(2m_{l-2}+1) \cdots (2m_1-1)1]^{\SO_{2n+1}},
\end{align*}
which is obtained from 
$$[k_s^2\cdots k_1^2(2m_l+1)(2m_{l-1}-1)(2m_{l-2}+1) \cdots (2m_1-1)1]$$ via 
replacing $k_s^2$ by $(k_s+1,k_s-1)$ if $k_s$ is even and  $k_j^{i,2}$ by $(k_j^i+1,k_j^i-1)$ for $1 \leq j \leq \frac{l-2}{2}$ and $1 \leq i \leq s_j$; and
replacing $k_l^{i,2}$ by $(k_l^i+1,k_l^i-1)$ for $1 \leq i \leq s_l$;
and finally replacing 
$k_0^{i,2}$ by $(k_0^i+1,k_0^i-1)$ for $1 \leq i \leq s_0$.
Hence \eqref{keylemma3equ1} still holds in this case. 

This completes the proof of the lemma. 
\end{proof}

The proof of Theorem \ref{main2} has been completed for $G_n=\SO_{2n+1}$.

\section{Proof of Theorem \ref{main2}, $G_n=\RO_{2n}$}

By the assumption of  Theorem \ref{main2}, $\sigma$ is of 
{\bf Type (I)} and is of the following form:

\begin{equation*}
\sigma = \times_{(\chi, m, \alpha) \in \textbf{e}}
v^{\alpha} \chi({\det}_m) \rtimes \sigma_{neg},
\end{equation*}
where $\sigma_{neg}$ is the unique irreducible negative unramified  subrepresentation
of the following induced representation
\begin{equation*}
\chi_1 ({\det}_{{n_1}}) \times \cdots \times \chi_t ({\det}_{{n_t}}) \rtimes \sigma_{sn},
\end{equation*}
with $\sigma_{sn}$ being the unique strongly negative unramified constituent of the following induced representation:
\begin{align*}
\begin{split}
& \nu^{\frac{m_{l-1}-m_l}{2}} 1_{{\det}_{m_{l-1}+m_l+1}}
\times
\nu^{\frac{m_{l-3}-m_{l-2}}{2}} 1_{{\det}_{m_{l-3}+m_{l-2}+1}}\\
&\qquad\times \cdots \times
\nu^{\frac{m_{1}-m_2}{2}} 1_{{\det}_{m_{1}+m_2+1}}
\rtimes 1_{\RO_0}.
\end{split}
\end{align*}
Also recall that $l$ is even, and $0 < m_1 < m_2 < \cdots < m_l$. 

As in Sections 5 and 6, by Lemma \ref{keylemma}, we have 
\begin{align*}
    &\frak{p}^m(\times_{(\chi, m, \alpha) \in \textbf{e}}
v^{\alpha} \chi({\det}_m) \times \chi_1 ({\det}_{{n_1}}) \times \cdots \times \chi_t ({\det}_{{n_t}}))\\
&\qquad=\{[(\prod_{(\chi, m, \alpha) \in \textbf{e}} m) (\prod_{i=1}^t n_i)]^t\}.
\end{align*}
By Theorem \ref{keylemma0}, Lemma \ref{keylemma}, and by \cite[Theorem 7.3.3]{CM93} on formula for induced nilpotent orbits, any $\ul{p} \in \frak{p}^m(\sigma)$ has the following upper bound
$$
\ul{p} \leq \left(2[(\prod_{(\chi, m, \alpha) \in \textbf{e}} m) (\prod_{i=1}^t n_i)]^t + \eta_{\frak{o}_{2k}, \frak{o}_{2k}}([\prod_{j=1}^l(2m_j+1)])\right)_{O_{2n}}
$$
with $2k=\sum_{i=1}^l (2m_i+1)$. 

To prove Theorem \ref{main2} in this case, it suffices to show the following lemma.

\begin{lem}\label{keylemma4}
The following identity
$$\eta_{\frak{o}_{2n}, \frak{o}_{2n}} (\ul{p}(\psi)) = \left(2[(\prod_{(\chi, m, \alpha) \in \textbf{e}} m) (\prod_{i=1}^t n_i)]^t + \eta_{\frak{o}_{2k}, \frak{o}_{2k}}([\prod_{j=1}^l(2m_j+1)])\right)_{\RO_{2n}}$$
holds with  $2k=\sum_{i=1}^l (2m_i+1)$.
\end{lem}

\begin{proof}
Recall that
$$\ul{p}(\psi)=[(\prod_{(\chi, m, \alpha) \in \textbf{e}} m^2) (\prod_{i=1}^t n_i^2)(\prod_{j=1}^l(2m_j+1))]$$
and
$$\eta_{\frak{o}_{2n}, \frak{o}_{2n}} (\ul{p}(\psi)) = (\ul{p}(\psi)^t)_{\RO_{2n}}.$$
Also recall that given any partition $\ul{p}=[p_r \cdots p_1]$ with $p_r \geq \cdots \geq p_1$, we have 
\begin{align*}
    \ul{p}^+&=[(p_r+1) \cdots p_1],\\
    \ul{p}^-&=[p_r \cdots (p_1-1)].
\end{align*}
By \cite[Lemma 3.3]{Ac03}, given a partition $\ul{p}$ of $2n$, 
if it is an orthogonal partition or its transpose is a symplectic partition, then 
$(\ul{p}^t)_{\RO_{2n}}=((\ul{p}^{+-})_{\Sp_{2n}})^t$. 
Hence we obtain that 
\begin{align*}
    &\left(2[(\prod_{(\chi, m, \alpha) \in \textbf{e}} m) (\prod_{i=1}^t n_i)]^t + \eta_{\frak{o}_{2k}, \frak{o}_{2k}}([\prod_{j=1}^l(2m_j+1)])\right)_{\RO_{2n}}\\
    &\quad=\left(2[(\prod_{(\chi, m, \alpha) \in \textbf{e}} m) (\prod_{i=1}^t n_i)]^t + ([\prod_{j=1}^l(2m_j+1)]^t)_{\RO_{2k}}\right)_{\RO_{2n}}\\
     &\quad\quad=\left(2[(\prod_{(\chi, m, \alpha) \in \textbf{e}} m) (\prod_{i=1}^t n_i)]^t + (([\prod_{j=1}^l(2m_j+1)]^{+-})_{\Sp_{2k}})^t\right)_{\RO_{2n}}.
\end{align*}
It is easy to see that
\begin{align*}
    &[(\prod_{(\chi, m, \alpha) \in \textbf{e}} m^2) (\prod_{i=1}^t n_i^2)(\prod_{j=1}^l(2m_j+1))]^t\\
   &\qquad\qquad= 2[(\prod_{(\chi, m, \alpha) \in \textbf{e}} m) (\prod_{i=1}^t n_i)]^t + ([\prod_{j=1}^l(2m_j+1)])^t
\end{align*}
is a partition of the following form
$$[p_l^1 \cdots p_l^{2m_1+1} (\prod_{j=1}^{l-1} p_{j}^1\cdots p_{j}^{2m_{l+1-j}-2m_{l-j}})p_0^1 \cdots p_0^{m_0}],$$
where $p_l^{i}$ with $1\leq i \leq 2m_1+1$, $p_{2j}^i$ with $1\leq i \leq 2m_{l+1-2j}-2m_{l-2j}$ and $1\leq j \leq \frac{l-2}{2}$, and $p_0^k$ with $1 \leq k \leq m_0$  
are all even; and $p_{2j+1}^i$ with $1\leq i \leq 2m_{l-2j}-2m_{l-2j-1}$ and $0\leq j \leq \frac{l-2}{2}$
are all odd; and finally $p_l^1 \geq \cdots \geq p_l^{2m_1+1}>p_{l-1}^1$, $p_j^1 \geq \cdots \geq p_j^{2m_{l+1-j}-2m_{l-j}} > p_{j-1}^1$, $1 \leq j \leq l-1$ with  $p_0^1 \geq \cdots \geq p_0^{m_0}$. 
Note that
\begin{align*}
\left([\prod_{j=1}^l(2m_j+1)]^{+-}\right)_{\Sp_{2k}}
&=
\left[(2m_l+2)\prod_{j=2}^{l-1}(2m_j+1)(2m_1)\right]_{\Sp_{2k}}\\
&= \left[(2m_l+2)\prod_{j=1}^{\frac{l-2}{2}}(2m_{2j+1})(2m_{2j}+2)(2m_1)\right].
\end{align*}
Then the partition 
\[
2[(\prod_{(\chi, m, \alpha) \in \textbf{e}} m) (\prod_{i=1}^t n_i)]^t + (([\prod_{j=1}^l(2m_j+1)]^{+-})_{\Sp_{2k}})^t
\]
is equal to the following partition
\begin{align*}
    &\,[p_l^1 \cdots p_l^{2m_1} (p_l^{2m_1+1}-1)\prod_{j=0}^{\frac{l-2}{2}} p_{2j+1}^1\cdots p_{2j+1}^{2m_{l-2j}-2m_{l-2j-1}}\\
    &\qquad\prod_{j=1}^{\frac{l-2}{2}}(p_{2j}^1+1)p_{2j}^2\cdots p_{2j}^{2m_{l+1-2j}-2m_{l-2j}-1}(p_{2j}^{2m_{l+1-2j}-2m_{l-2j}}-1)\\
    &\qquad\qquad(p_0^1+1) p_0^2 \cdots p_0^{m_0}].
\end{align*}
Following the recipe on carrying out the $\RO_{2n}$-collapse (\cite[Lemma 6.3.8]{CM93}), we obtain that the partition 
\[
\left([(\prod_{(\chi, m, \alpha) \in \textbf{e}} m^2) (\prod_{i=1}^t n_i^2)(\prod_{j=1}^l(2m_j+1))]^t\right)_{\RO_{2n}}
\]
is equal to the 
following partition
\begin{align*}
    &[(p_l^1 \cdots p_l^{2m_1})_\RO (p_l^{2m_1+1}-1)\prod_{j=0}^{\frac{l-2}{2}} p_{2j+1}^1\cdots p_{2j+1}^{2m_{l-2j}-2m_{l-2j-1}}\\
    &\qquad\prod_{j=1}^{\frac{l-2}{2}}(p_{2j}^1+1)(p_{2j}^2\cdots p_{2j}^{2m_{l+1-2j}-2m_{l-2j}-1})_\RO(p_{2j}^{2m_{l+1-2j}-2m_{l-2j}}-1)\\
    &\qquad\qquad(p_0^1+1) (p_0^2 \cdots p_0^{m_0})_\RO],
\end{align*}
and the partition 
\[
\left(2[(\prod_{(\chi, m, \alpha) \in \textbf{e}} m) (\prod_{i=1}^t n_i)]^t + (([\prod_{j=1}^l(2m_j+1)]^{+-})_{\Sp_{2k}})^t\right)_{\RO_{2n}}
\]
can be written as 
\begin{align*}
    & \bigg(p_l^1 \cdots p_l^{2m_1} (p_l^{2m_1+1}-1)\prod_{j=0}^{\frac{l-2}{2}} p_{2j+1}^1\cdots p_{2j+1}^{2m_{l-2j}-2m_{l-2j-1}}\\
    &\qquad\prod_{j=1}^{\frac{l-2}{2}}(p_{2j}^1+1)p_{2j}^2\cdots p_{2j}^{2m_{l+1-2j}-2m_{l-2j}-1}(p_{2j}^{2m_{l+1-2j}-2m_{l-2j}}-1)\\
    &\qquad\qquad(p_0^1+1) p_0^2\cdots p_0^{m_0}\bigg)_{\RO_{2n}}
\end{align*}
which is equal to 
\begin{align*}
    &[(p_l^1 \cdots p_l^{2m_1})_\RO (p_l^{2m_1+1}-1)\prod_{j=0}^{\frac{l-2}{2}} p_{2j+1}^1\cdots p_{2j+1}^{2m_{l-2j}-2m_{l-2j-1}}\\
    &\qquad \prod_{j=1}^{\frac{l-2}{2}}(p_{2j}^1+1)(p_{2j}^2\cdots p_{2j}^{2m_{l+1-2j}-2m_{l-2j}-1})_\RO(p_{2j}^{2m_{l+1-2j}-2m_{l-2j}}-1)\\
    &\qquad\qquad(p_0^1+1) (p_0^2 \cdots p_0^{m_0})_\RO].
\end{align*}
Hence we obtain that
$$\eta_{\frak{o}_{2n}, \frak{o}_{2n}} (\ul{p}(\psi)) = \left(2[(\prod_{(\chi, m, \alpha) \in \textbf{e}} m) (\prod_{i=1}^t n_i)]^t + \eta_{\frak{o}_{2k}, \frak{o}_{2k}}([\prod_{j=1}^l(2m_j+1)])\right)_{\RO_{2n}}.$$
This completes the proof of the lemma. 
\end{proof}

The proof of Theorem \ref{main2} has been completed for $G_n=\RO_{2n}$.

\section{On the wave front set of unramified unitary representations}

In this last section, we study the wave front set of the unramified unitary representations for split classical groups $G_n=\Sp_{2n}, \SO_{2n+1}, \RO_{2n}$. Under assumptions on the leading orbits in the wave front set of negative representations, we determine the set $\frak{p}^m(\pi)$ for general unramified unitary representations. This reduction has its own interests. 

Assume that $\pi$ is any irreducible unramified unitary representation of $G_n(F)$ as in Theorem 
\ref{thm10},
$$\pi= \times_{(\chi, m, \alpha) \in \textbf{e}}
v^{\alpha} \chi({\det}_m) \rtimes \sigma_{neg},$$
where $\sigma_{neg}$ is a negative representation of $G_{n^*}(F)$, and
\begin{align*}
\rm{Jord}(\sigma_{neg}) &= \rm{Jord}(\sigma_{sn}) \cup \{(\chi_i,n_i),
(\chi_i^{-1},n_i) | 1 \leq i \leq t\}.
\end{align*}
Here $\rm{Jord}(\sigma_{sn})$ is equal to 
\begin{align*}
 \left\{(\lambda_0, 2n_1+1), \ldots, (\lambda_0, 2n_k+1), (1_{\GL_1}, 2m_1+1),
\ldots, (1_{\GL_1}, 2m_l+1)\right\},
\end{align*}
when $G_{n^*}=\Sp_{2{n^*}}, \RO_{2{n^*}}$; and is equal to 
\begin{align*}
\left\{(\lambda_0, 2n_1), \ldots, (\lambda_0, 2n_k), (1_{\GL_1}, 2m_1),
\ldots, (1_{\GL_1}, 2m_l)\right\},
\end{align*}
when $G_{n^*}=\SO_{2{n^*}+1}$, as in Section \ref{unramified unitary dual}. 

We have the following conjecture on the maximal partitions in the wave front set of negative representations. 

\begin{conj}\label{conj-neg}
\begin{align*}
   \frak{p}^m(\sigma_{neg})= \left\{\eta_{\frak{g}_{{n^*}}^\vee, \frak{g}_{{n^*}}} ([(\prod_{j=1}^t n_j^2) (\prod_{i=1}^l (2m_i+1)\prod_{s=1}^k (2n_s+1))])\right\}, 
\end{align*}
when $G_{n^*}=\Sp_{2{n^*}}, \RO_{2{n^*}}$;
\begin{align*}
\frak{p}^m(\sigma_{neg})= \left\{\eta_{\frak{g}_{{n^*}}^\vee, \frak{g}_{{n^*}}} ([(\prod_{j=1}^t n_j^2) (\prod_{i=1}^l (2m_i)\prod_{s=1}^k (2n_s))])\right\},
\end{align*}
when $G_{n^*}=\SO_{2{n^*}+1}$.
\end{conj}

Based on Conjecture \ref{conj-neg}, we obtain the explicit description on the maximal partitions in the wave-front set of 
general irreducible unramified unitary  representations $\pi$ of 
$G_n(F)$. 

\begin{thm}\label{main3}
Assume Conjecture \ref{conj-neg} is true. 
For any irreducible unramified unitary  representation $\pi$ of 
$G_n(F)$, the maximal partitions in the wave-front set $\Fp(\pi)$ are given as follows:
\begin{align*}
   \frak{p}^m(\pi)= \left\{\eta_{\frak{g}_{{n}}^\vee, \frak{g}_{{n}}} ([(\prod_{(\chi,m,\alpha) \in \textbf{e}} m^2)(\prod_{j=1}^t n_j^2) (\prod_{i=1}^l (2m_i+1)\prod_{s=1}^k (2n_s+1))])\right\}
\end{align*}
when $G_{n}=\Sp_{2{n}}, \RO_{2{n}}$; and 
\begin{align*}
\frak{p}^m(\pi)= \left\{\eta_{\frak{g}_{{n}}^\vee, \frak{g}_{{n}}} ([(\prod_{(\chi,m,\alpha) \in \textbf{e}} m^2)(\prod_{j=1}^t n_j^2) (\prod_{i=1}^l (2m_i)\prod_{s=1}^k (2n_s))])\right\}
\end{align*}
when $G_{n}=\SO_{2{n}+1}$.
\end{thm}

We remark that Ciubotaru, Mason-Brown, and Okada (\cite{CMO21}) recently computed the maximal orbits in the wave front set of irreducible Iwahori-spherical representations
of split reductive $p$-adic groups
with ``real infinitesimal characters", which partially proved
Conjecture \ref{conj-neg} and Theorem \ref{main3}. This provides evidence for Conjecture \ref{conj-neg}.

By Lemma \ref{keylemma}, we have 
\begin{align*}
\frak{p}^m(\times_{(\chi, m, \alpha) \in \textbf{e}}
v^{\alpha} \chi({\det}_m))
=\left\{+_{(\chi, m, \alpha) \in \textbf{e}} [1^m]\right\}
=
\{[(\prod_{(\chi, m, \alpha) \in \textbf{e}} m)]^t\}.
\end{align*}
By Lemma \ref{keylemma}, and by \cite[Theorem 7.3.3]{CM93} on formula for induced nilpotent orbits, we obtain that 
$$\frak{p}^m(\pi)=\{ (2[(\prod_{(\chi, m, \alpha) \in \textbf{e}} m)]^t + \ul{p}_{\sigma_{neg}})_{G_{n}} \big| \ul{p}_{\sigma_{neg}} \in \frak{p}^m(\sigma_{neg})\}.$$
Hence, by the assumption, to prove Theorem \ref{main3}, it suffices to show the following lemma which will be proved case-by-case in the following subsections. 

\begin{lem}\label{keylemma5}
The following identities hold:
\begin{align*}
&\eta_{\frak{g}_{{n}}^\vee, \frak{g}_{{n}}} ([(\prod_{(\chi,m,\alpha) \in \textbf{e}} m^2)(\prod_{j=1}^t n_j^2) (\prod_{i=1}^l (2m_i+1)\prod_{s=1}^k (2n_s+1))])\\
&\quad= \left(2[(\prod_{(\chi, m, \alpha) \in \textbf{e}} m)]^t + \eta_{\frak{g}_{{n^*}}^\vee, \frak{g}_{{n^*}}} [(\prod_{j=1}^t n_j^2) (\prod_{i=1}^l (2m_i+1)\prod_{s=1}^k (2n_s+1))]\right)_{G_{n}}
\end{align*}
when $G_{n}=\Sp_{2{n}}, \RO_{2{n}}$; and 
\begin{align*}
&\eta_{\frak{g}_{{n}}^\vee, \frak{g}_{{n}}} ([(\prod_{(\chi,m,\alpha) \in \textbf{e}} m^2)(\prod_{j=1}^t n_j^2) (\prod_{i=1}^l (2m_i)\prod_{s=1}^k (2n_s))])\\
&\qquad= \left(2[(\prod_{(\chi, m, \alpha) \in \textbf{e}} m)]^t + \eta_{\frak{g}_{{n^*}}^\vee, \frak{g}_{{n^*}}} ([(\prod_{j=1}^t n_j^2) (\prod_{i=1}^l (2m_i)\prod_{s=1}^k (2n_s))])\right)_{G_{n}}
\end{align*}
when $G_{n}=\SO_{2{n}+1}$.
\end{lem}

\subsection{Proof of Lemma \ref{keylemma5}, $G_n=\Sp_{2n}$}

By similar arguments as in the proof of the $\Sp_{2n}$-case of Lemma \ref{keylemma2}, we only need to show that 
\begin{align}\label{keylemma5equ1}
\begin{split}
&\left([(\prod_{(\chi,m,\alpha) \in \textbf{e}} m^2)(\prod_{j=1}^t n_j^2) (\prod_{i=1}^l (2m_i+1)\prod_{s=1}^k (2n_s+1))]^-\right)_{\Sp_{2n}}\\
&\quad=\left((\prod_{(\chi, m, \alpha) \in \textbf{e}} m^2) ([(\prod_{j=1}^t n_j^2) (\prod_{i=1}^l (2m_i+1)\prod_{s=1}^k (2n_s+1))]^-)_{\Sp_{2n^*}}\right)^{\Sp_{2n}}.
\end{split}
\end{align}
For any given partition $\ul{p}=[p_r \cdots p_1]$ with $p_r \geq \cdots \geq p_1$, we recall that $\ul{p}^-=[p_r \cdots (p_1-1)]$.
Rewrite the partition $[\prod_{(\chi, m, \alpha) \in \textbf{e}} m^2]$ as
$[p_u^2p_{u-1}^2\cdots p_1^2]$ with $p_u \geq p_{u-1} \geq \cdots \geq p_1$; and 
$[\prod_{i=1}^t n_i^2]$ as
$[q_{v}^2q_{v-1}^2\cdots q_1^2]$ with $q_{v} \geq q_{v-1} \geq \cdots \geq q_1$.
And rewrite $[(\prod_{i=1}^l (2m_i+1)\prod_{s=1}^k (2n_s+1))]$ as $[\prod_{w=1}^{l+k} (2r_w+1)]$
with $r_{l+k} \geq r_{l+k-1} \geq \cdots \geq r_{1} > 0$. Then, \eqref{keylemma5equ1} becomes
\begin{align}\label{keylemma5equ4}
\begin{split}
&\left([(\prod_{i=1}^u p_1^2)(\prod_{j=1}^v q_j^2) \prod_{w=1}^{l+k} (2r_w+1)]^-\right)_{\Sp_{2n}}\\
&\qquad=\left((\prod_{i=1}^u p_1^2)([(\prod_{j=1}^v q_j^2) \prod_{w=1}^{l+k} (2r_w+1)]^-)_{\Sp_{2n^*}}\right)^{\Sp_{2n}}.
\end{split}
\end{align}
To proceed, we separate into the  following cases:
\begin{enumerate}
    \item When $q_1 \geq 2r_1+1$, we have (a) $p_1 \geq 2r_1+1$ 
    and (b) $p_1 < 2r_1+1$.
    \item When $q_1 < 2r_1+1$, we have (a) $p_1 \geq q_1$, and 
    (b)$p_1 < q_1$.
\end{enumerate}
In each case, for $1 \leq z \leq \frac{l+k-1}{2}$, if $2r_{2z+1}+1> 2r_{2z}+1$, 
we list all the different odd $p_i$'s, $q_j$'s between $2r_{2z+1}+1$ and $2r_{2z}+1$ as 
\begin{align*}
   &\, 2r_{2z+1}+1 > p_z^1 > p_z^2 > \cdots > p_z^{x_z} > 2r_{2z}+1,\\
   &\, 2r_{2z+1}+1 > q_z^1 > q_z^2 > \cdots > q_z^{y_z} > 2r_{2z}+1.
\end{align*}

\quad

{\bf Case (1-a):} $q_1 \geq 2r_1+1$, $p_1 \geq 2r_1+1$. We have 
\begin{align*}
\begin{split}
&\left([(\prod_{i=1}^u p_1^2)(\prod_{j=1}^v q_j^2) \prod_{w=1}^{l+k} (2r_w+1)]^-\right)_{\Sp_{2n}}\\
&\qquad=\left[(\prod_{i=1}^u p_1^2)(\prod_{j=1}^v q_j^2) \prod_{w=2}^{l+k} (2r_w+1)(2r_1)\right]_{\Sp_{2n}}.
\end{split}
\end{align*}
The collapse $[(\prod_{i=1}^u p_1^2)(\prod_{j=1}^v q_j^2) \prod_{w=2}^{l+k} (2r_w+1)(2r_1)]_{\Sp_{2n}}$ can be obtained from 
\[
[(\prod_{i=1}^u p_1^2)(\prod_{j=1}^v q_j^2) \prod_{w=2}^{l+k} (2r_w+1)(2r_1)]
\]
via  replacing $(2r_{2z+1}+1, 2r_{2z}+1)$ by $(2r_{2z+1}, 2r_{2z}+2)$, $p_z^{i,2}$ by $(p_z^i+1,p_z^i-1)$, and $q_z^{j,2}$ by $(q_z^j+1,q_z^j-1)$, for $1 \leq z \leq \frac{l+s-1}{2}$, $1 \leq i \leq x_z$, and $1 \leq j \leq y_z$, whenever $2r_{2z+1}+1> 2r_{2z}+1$. 
On the other hand, we have 
\begin{align}\label{keylemma5equ3}
\begin{split}
&\left((\prod_{i=1}^u p_1^2)([(\prod_{j=1}^v q_j^2) \prod_{w=1}^{l+k} (2r_w+1)]^-)_{\Sp_{2n^*}}\right)^{\Sp_{2n}}\\
    &\qquad=\left((\prod_{i=1}^u p_1^2)[(\prod_{j=1}^v q_j^2) \prod_{w=2}^{l+k} (2r_w+1)(2r_1)]_{\Sp_{2n^*}}\right)^{\Sp_{2n}}. 
    \end{split}
\end{align}
Then $[(\prod_{j=1}^v q_j^2) \prod_{w=2}^{l+k} (2r_w+1)(2r_1)]_{Sp_{2n^*}}$ can be obtained from 
$$[(\prod_{j=1}^v q_j^2) \prod_{w=2}^{l+k} (2r_w+1)(2r_1)]$$
via replacing $(2r_{2z+1}+1, 2r_{2z}+1)$ by $(2r_{2z+1},  2r_{2z}+2)$ and $q_z^{j,2}$ by $(q_z^j+1,q_z^j-1)$, for $1 \leq z \leq \frac{l+s-1}{2}$ and $1 \leq j \leq y_z$, whenever $2r_{2z+1}+1> 2r_{2z}+1$.
And the partition 
\[
\left((\prod_{i=1}^u p_1^2)[(\prod_{j=1}^v q_j^2) \prod_{w=2}^{l+k} (2r_w+1)(2r_1)]_{\Sp_{2n^*}}\right)^{\Sp_{2n}}
\]
can be obtained from $[(\prod_{i=1}^u p_1^2)[(\prod_{j=1}^v q_j^2) \prod_{w=2}^{l+k} (2r_w+1)(2r_1)]_{\Sp_{2n^*}}]$ via replacing $p_z^{i,2}$ by $(p_z^i+1,p_z^i-1)$ for $1 \leq z \leq \frac{l+s-1}{2}$ and $1 \leq i \leq x_z$, whenever $2r_{2z+1}+1> 2r_{2z}+1$.
Hence \eqref{keylemma5equ4} holds in this case. 

\quad

{\bf Case (1-b):} $q_1 \geq 2r_1+1$, $p_1 < 2r_1+1$. We have 
\begin{align}\label{keylemma5equ2}
\begin{split}
&\left([(\prod_{i=1}^u p_1^2)(\prod_{j=1}^v q_j^2) \prod_{w=1}^{l+k} (2r_w+1)]^-\right)_{\Sp_{2n}}\\
&\qquad=\left[(\prod_{i=2}^u p_1^2)(\prod_{j=1}^v q_j^2) \prod_{w=1}^{l+k} (2r_w+1)p_1(p_1-1)\right]_{\Sp_{2n}}.
\end{split}
\end{align}
To carry out the $\Sp_{2n}$-collapse, 
we also need to list all the different odd $p_i$'s, $q_j$'s between $2r_{1}+1$ and $p_1$ as 
\begin{align*}
   &\, 2r_{1}+1 > p_0^1 > p_0^2 > \cdots > p_0^{x_0} > p_1,\\
   &\, 2r_{1}+1 > q_0^1 > q_0^2 > \cdots > q_0^{y_0} > p_1.
\end{align*} 
Then $[(\prod_{i=2}^u p_1^2)(\prod_{j=1}^v q_j^2) \prod_{w=1}^{l+k} (2r_w+1)p_1(p_1-1)]_{\Sp_{2n}}$ can be obtained from 
\[
[(\prod_{i=2}^u p_1^2)(\prod_{j=1}^v q_j^2) \prod_{w=1}^{l+k} (2r_w+1)p_1(p_1-1)]
\]
via replacing $(2r_{2z+1}+1, 2r_{2z}+1)$ by $(2r_{2z+1}, 2r_{2z}+2)$, $p_z^{i,2}$ by $(p_z^i+1,p_z^i-1)$, and $q_z^{j,2}$ by $(q_z^j+1,q_z^j-1)$, for $1 \leq z \leq \frac{l+s-1}{2}$, $1 \leq i \leq x_z$, and $1 \leq j \leq y_z$, whenever $2r_{2z+1}+1> 2r_{2z}+1$; and replacing $(2r_{1}+1, p_1-1)$ by $(2r_{1}, p_1)$ if $p_1$ is even and $(2r_{1}+1, p_1)$ by $(2r_{1}, p_1+1)$ if $p_1$ is odd; and finally replacing 
$p_0^{i,2}$ by $(p_0^i+1,p_0^i-1)$ and $q_0^{j,2}$ by $(q_0^j+1,q_0^j-1)$, for $1 \leq i \leq x_0$ and $1 \leq j \leq y_0$. 

On the other hand, as in {\bf Case (1-a)}, we still have \eqref{keylemma5equ3}. 
Then $[(\prod_{j=1}^v q_j^2) \prod_{w=2}^{l+k} (2r_w+1)(2r_1)]_{\Sp_{2n^*}}$ can be obtained from 
$$[(\prod_{j=1}^v q_j^2) \prod_{w=2}^{l+k} (2r_w+1)(2r_1)]$$
via replacing $(2r_{2z+1}+1, 2r_{2z}+1)$ by $(2r_{2z+1}, 2r_{2z}+2)$ and $q_z^{j,2}$ by $(q_z^j+1,q_z^j-1)$, for $1 \leq z \leq \frac{l+s-1}{2}$ and $1 \leq j \leq y_z$, whenever $2r_{2z+1}+1> 2r_{2z}+1$.
And the partition 
\[
\left((\prod_{i=1}^u p_1^2)[(\prod_{j=1}^v q_j^2) \prod_{w=2}^{l+k} (2r_w+1)(2r_1)]_{\Sp_{2n^*}}\right)^{\Sp_{2n}}
\]
can be obtained from $[(\prod_{i=1}^u p_1^2)[(\prod_{j=1}^v q_j^2) \prod_{w=2}^{l+k} (2r_w+1)(2r_1)]_{\Sp_{2n^*}}]$ via replacing $p_z^{i,2}$ by $(p_z^i+1,p_z^i-1)$ for $1 \leq z \leq \frac{l+s-1}{2}$ and $1 \leq i \leq x_z$, whenever $2r_{2z+1}+1> 2r_{2z}+1$; and then replacing $p_1^2$ by $(p_1+1,p_1-1)$ if $p_1$ is odd, 
$p_0^{i,2}$ by $(p_0^i+1,p_0^i-1)$, and $q_0^{j,2}$ by $(q_0^j+1,q_0^j-1)$ for  $1 \leq i \leq x_0$ and $1 \leq j \leq y_0$.
Hence \eqref{keylemma5equ4} holds in this case. 

\quad

{\bf Case (2-a):} $q_1 < 2r_1+1$, $p_1 \geq q_1$. We have 
\begin{align*}
\begin{split}
&\left([(\prod_{i=1}^u p_1^2)(\prod_{j=1}^v q_j^2) \prod_{w=1}^{l+k} (2r_w+1)]^-\right)_{\Sp_{2n}}\\
&\qquad=\left[(\prod_{i=1}^u p_1^2)(\prod_{j=2}^v q_j^2) \prod_{w=1}^{l+k} (2r_w+1)q_1(q_1-1)\right]_{\Sp_{2n}}.
\end{split}
\end{align*}
To carry out the $\Sp_{2n}$-collapse, 
we also need to list all the different odd $p_i$'s, $q_j$'s between $2r_{1}+1$ and $q_1$ as 
\begin{align*}
   &\, 2r_{1}+1 > p_0^1 > p_0^2 > \cdots > p_0^{x_0} > q_1,\\
   &\, 2r_{1}+1 > q_0^1 > q_0^2 > \cdots > q_0^{y_0} > q_1.
\end{align*} 
Then $[(\prod_{i=1}^u p_1^2)(\prod_{j=2}^v q_j^2) \prod_{w=1}^{l+k} (2r_w+1)q_1(q_1-1)]_{\Sp_{2n}}$ can be obtained from 
\[
[(\prod_{i=1}^u p_1^2)(\prod_{j=2}^v q_j^2) \prod_{w=1}^{l+k} (2r_w+1)q_1(q_1-1)]
\]
via replacing $(2r_{2z+1}+1, 2r_{2z}+1)$ by $(2r_{2z+1}, 2r_{2z}+2)$, $p_z^{i,2}$ by $(p_z^i+1,p_z^i-1)$, and $q_z^{j,2}$ by $(q_z^j+1,q_z^j-1)$, for $1 \leq z \leq \frac{l+s-1}{2}$, $1 \leq i \leq x_z$, and $1 \leq j \leq y_z$, whenever $2r_{2z+1}+1> 2r_{2z}+1$; and replacing $(2r_{1}+1, q_1-1)$ by $(2r_{1}, q_1)$ if $q_1$ is even, $(2r_{1}+1, q_1)$ by $(2r_{1}, q_1+1)$ if $q_1$ is odd; and finally replacing
$p_0^{i,2}$ by $(p_0^i+1,p_0^i-1)$ and $q_0^{j,2}$ by $(q_0^j+1,q_0^j-1)$ for $1 \leq i \leq x_0$ and $1 \leq j \leq y_0$.
On the other hand, we have 
\begin{align}\label{keylemma5equ5}
\begin{split}
&\left((\prod_{i=1}^u p_1^2)([(\prod_{j=1}^v q_j^2) \prod_{w=1}^{l+k} (2r_w+1)]^-)_{\Sp_{2n^*}}\right)^{\Sp_{2n}}\\
    &\qquad=\left((\prod_{i=1}^u p_1^2)[(\prod_{j=2}^v q_j^2) \prod_{w=1}^{l+k} (2r_w+1)q_1(q_1-1)]_{\Sp_{2n^*}}\right)^{\Sp_{2n}}. 
    \end{split}
\end{align}
Then $[(\prod_{j=2}^v q_j^2) \prod_{w=1}^{l+k} (2r_w+1)q_1(q_1-1)]_{\Sp_{2n^*}}$ can be obtained from 
$$[(\prod_{j=2}^v q_j^2) \prod_{w=1}^{l+k} (2r_w+1)q_1(q_1-1)]$$
via replacing $(2r_{2z+1}+1, 2r_{2z}+1)$ by $(2r_{2z+1}, 2r_{2z}+2)$ and $q_z^{j,2}$ by $(q_z^j+1,q_z^j-1)$ for $1 \leq z \leq \frac{l+s-1}{2}$ and $1 \leq j \leq y_z$, whenever $2r_{2z+1}+1> 2r_{2z}+1$; and then replacing $(2r_{1}+1, q_1-1)$ by $(2r_{1}, q_1)$ if $q_1$ is even and $(2r_{1}+1, q_1)$ by $(2r_{1}, q_1+1)$ if $q_1$ is odd, and $q_0^{j,2}$ by $(q_0^j+1,q_0^j-1)$ for $1 \leq j \leq y_0$.
And the partition 
\[
\left((\prod_{i=1}^u p_1^2)[(\prod_{j=2}^v q_j^2) \prod_{w=1}^{l+k} (2r_w+1)q_1(q_1-1)]_{\Sp_{2n^*}}\right)^{\Sp_{2n}}
\]
can be obtained from $[(\prod_{i=1}^u p_1^2)[(\prod_{j=2}^v q_j^2) \prod_{w=1}^{l+k} (2r_w+1)q_1(q_1-1)]_{Sp_{2n^*}}]$ via replacing $p_z^{i,2}$ by $(p_z^i+1,p_z^i-1)$ for $1 \leq z \leq \frac{l+s-1}{2}$ and $1 \leq i \leq x_z$, whenever $2r_{2z+1}+1> 2r_{2z}+1$; and then replacing $p_0^{i,2}$ by $(p_0^i+1,p_0^i-1)$ for  $1 \leq i \leq x_0$.
Hence \eqref{keylemma5equ4} holds in this case.

\quad

{\bf Case (2-b):} $q_1 < 2r_1+1$, $p_1 < q_1$. We have 
\begin{align*}
\begin{split}
&\left([(\prod_{i=1}^u p_1^2)(\prod_{j=1}^v q_j^2) \prod_{w=1}^{l+k} (2r_w+1)]^-\right)_{\Sp_{2n}}\\
&\qquad=\left[(\prod_{i=2}^u p_1^2)(\prod_{j=1}^v q_j^2) \prod_{w=1}^{l+k} (2r_w+1)p_1(p_1-1)\right]_{\Sp_{2n}}.
\end{split}
\end{align*}
To carry out the $\Sp_{2n}$-collapse, 
we also need to list all the different odd $p_i$'s, $q_j$'s between $2r_{1}+1$ and $p_1$ as 
\begin{align*}
   &\, 2r_{1}+1 > p_0^1 > p_0^2 > \cdots > p_0^{x_0} > p_1,\\
   &\, 2r_{1}+1 > q_0^1 > q_0^2 > \cdots > q_0^{y_0} > p_1.
\end{align*} 
Then $[(\prod_{i=2}^u p_1^2)(\prod_{j=1}^v q_j^2) \prod_{w=1}^{l+k} (2r_w+1)p_1(p_1-1)]_{\Sp_{2n}}$ can be obtained from 
\[
[(\prod_{i=2}^u p_1^2)(\prod_{j=1}^v q_j^2) \prod_{w=1}^{l+k} (2r_w+1)p_1(p_1-1)]
\]
via replacing $(2r_{2z+1}+1, 2r_{2z}+1)$ by $(2r_{2z+1}, 2r_{2z}+2)$, $p_z^{i,2}$ by $(p_z^i+1,p_z^i-1)$, and $q_z^{j,2}$ by $(q_z^j+1,q_z^j-1)$, for $1 \leq z \leq \frac{l+s-1}{2}$, $1 \leq i \leq x_z$, and $1 \leq j \leq y_z$, whenever $2r_{2z+1}+1> 2r_{2z}+1$; and then replacing $(2r_{1}+1, p_1-1)$ by $(2r_{1}, p_1)$ if $p_1$ is even, $(2r_{1}+1, p_1)$ by $(2r_{1}, p_1+1)$ if $p_1$ is odd, 
and also $p_0^{i,2}$ by $(p_0^i+1,p_0^i-1)$ and $q_0^{j,2}$ by $(q_0^j+1,q_0^j-1)$ for  $1 \leq i \leq x_0$ and $1 \leq j \leq y_0$.

On the other hand, we still have \eqref{keylemma5equ5}.
Then the partition 
\[
\left[(\prod_{j=2}^v q_j^2) \prod_{w=1}^{l+k} (2r_w+1)q_1(q_1-1)\right]_{\Sp_{2n^*}}
\]
can be obtained from 
$$[(\prod_{j=2}^v q_j^2) \prod_{w=1}^{l+k} (2r_w+1)q_1(q_1-1)]$$
via replacing $(2r_{2z+1}+1, 2r_{2z}+1)$ by $(2r_{2z+1}, 2r_{2z}+2)$ and $q_z^{j,2}$ by $(q_z^j+1,q_z^j-1)$ for $1 \leq z \leq \frac{l+s-1}{2}$ and $1 \leq j \leq y_z$, whenever $2r_{2z+1}+1> 2r_{2z}+1$; and then replacing $(2r_{1}+1, q_1-1)$ by $(2r_{1}, q_1)$ if $q_1$ is even, $(2r_{1}+1, q_1)$ by $(2r_{1}, q_1+1)$ if $q_1$ is odd, and also $q_0^{j,2}$ by $(q_0^j+1,q_0^j-1)$ if $q_0^j \neq q_1$ and $1 \leq j \leq y_0$.
And the partition 
\[
\left((\prod_{i=1}^u p_1^2)[(\prod_{j=2}^v q_j^2) \prod_{w=1}^{l+k} (2r_w+1)q_1(q_1-1)]_{\Sp_{2n^*}}\right)^{\Sp_{2n}}
\]
can be obtained from $[(\prod_{i=1}^u p_1^2)[(\prod_{j=2}^v q_j^2) \prod_{w=1}^{l+k} (2r_w+1)q_1(q_1-1)]_{\Sp_{2n^*}}]$ via replacing $p_z^{i,2}$ by $(p_z^i+1,p_z^i-1)$ for $1 \leq z \leq \frac{l+s-1}{2}$ and $1 \leq i \leq x_z$, whenever $2r_{2z+1}+1> 2r_{2z}+1$; and then replacing 
$p_1^2$ by $(p_1+1,p_1-1)$ if $p_1$ is odd, 
$p_0^{i,2}$ by $(p_0^i+1,p_0^i-1)$ for $1 \leq i \leq x_0$.
Hence \eqref{keylemma5equ1} holds in this case.

The proof of Lemma \ref{keylemma5} has been completed for $G_n=\Sp_{2n}$.

\subsection{Proof of Lemma \ref{keylemma5}, $G_n=\SO_{2n+1}$}

By similar arguments as in the proof of the $\SO_{2n+1}$-case of Lemma \ref{keylemma2}, we only need to show that 
\begin{align}\label{keylemma5equ21}
&\left([(\prod_{(\chi,m,\alpha) \in \textbf{e}} m^2)(\prod_{j=1}^t n_j^2) (\prod_{i=1}^l (2m_i)\prod_{s=1}^k (2n_s))]^+\right)_{\SO_{2n+1}}\\
&\quad\qquad=\left((\prod_{(\chi, m, \alpha) \in \textbf{e}} m^2) ([(\prod_{j=1}^t n_j^2) (\prod_{i=1}^l (2m_i)\prod_{s=1}^k (2n_s))]^+)_{\SO_{2n^*+1}}\right)^{\SO_{2n+1}}.\nonumber
\end{align}
For any given partition $\ul{p}=[p_r \cdots p_1]$ with $p_r \geq \cdots \geq p_1$, recall that $\ul{p}^+=[(p_r+1) \cdots p_1]$.
Rewrite the partition $[\prod_{(\chi, m, \alpha) \in \textbf{e}} m^2]$ as
$[p_u^2p_{u-1}^2\cdots p_1^2]$ with $p_u \geq p_{u-1} \geq \cdots \geq p_1$; and the partition 
$[\prod_{i=1}^t n_i^2]$ as
$[q_{v}^2q_{v-1}^2\cdots q_1^2]$ with $q_{v} \geq q_{v-1} \geq \cdots \geq q_1$.
And rewrite the partition $[(\prod_{i=1}^l (2m_i)\prod_{s=1}^k (2n_s))]$ as $[\prod_{w=1}^{l+k} (2r_w)]$
with $r_{l+k} \geq r_{l+k-1} \geq \cdots \geq r_{1} \geq 0$. Then, \eqref{keylemma5equ21} becomes
\begin{align}\label{keylemma5equ22}
\begin{split}
&\left([(\prod_{i=1}^u p_1^2)(\prod_{j=1}^v q_j^2) \prod_{w=1}^{l+k} (2r_w)]^+\right)_{\SO_{2n+1}}\\
&\qquad= \left((\prod_{i=1}^u p_1^2)([(\prod_{j=1}^v q_j^2) \prod_{w=1}^{l+k} (2r_w)]^+)_{\SO_{2n^*+1}}\right)^{\SO_{2n+1}}.
\end{split}
\end{align}
To proceed, we separate into the  following cases:
\begin{enumerate}
    \item When $q_v \leq 2r_{l+k}$, we have (a) $p_u \leq 2r_{l+k}$ and (b) $p_u > 2r_{l+k}$. 
    \item When $q_v > 2r_{l+k}$, we have (a) $p_u \leq q_v$ and 
    (b) $p_u > q_v$.
\end{enumerate}
In each case, for $1 \leq z \leq \frac{l+k-2}{2}$, if $2r_{2z+1}> 2r_{2z}$, 
we list all the different even $p_i$'s, $q_j$'s between $2r_{2z+1}$ and $2r_{2z}$ as 
\begin{align*}
   &\, 2r_{2z+1} > p_z^1 > p_z^2 > \cdots > p_z^{x_z} > 2r_{2z},\\
   &\, 2r_{2z+1} > q_z^1 > q_z^2 > \cdots > q_z^{y_z} > 2r_{2z}.
\end{align*}
If $2r_1\neq0$, we also list all the different even $p_i$'s, $q_j$'s between $2r_{1}$ and $0$ as 
\begin{align*}
   &\, 2r_{1} > p_0^1 > p_0^2 > \cdots > p_0^{x_0} > 0,\\
   &\, 2r_{1} > q_0^1 > q_0^2 > \cdots > q_0^{y_0} > 0.
\end{align*}

\quad

{\bf Case (1-a):} $q_v \leq 2r_{l+k}$, $p_u \leq 2r_{l+k}$. We have 
\begin{align*}
\begin{split}
& \left([(\prod_{i=1}^u p_1^2)(\prod_{j=1}^v q_j^2) \prod_{w=1}^{l+k} (2r_w)]^+\right)_{\SO_{2n+1}}\\
&\qquad= \left[(\prod_{i=1}^u p_1^2)(\prod_{j=1}^v q_j^2) (2r_{l+k}+1)\prod_{w=1}^{l+k-1} (2r_w)\right]_{\SO_{2n+1}}.
\end{split}
\end{align*}
The collapse $[(\prod_{i=1}^u p_1^2)(\prod_{j=1}^v q_j^2) (2r_{l+k}+1)\prod_{w=1}^{l+k-1} (2r_w)]_{\SO_{2n+1}}$ can be obtained from 
\[
[(\prod_{i=1}^u p_1^2)(\prod_{j=1}^v q_j^2) (2r_{l+k}+1)\prod_{w=1}^{l+k-1} (2r_w)]
\]
via  replacing $(2r_{2z+1}, 2r_{2z})$ by $(2r_{2z+1}-1, 2r_{2z}+1)$, $p_z^{i,2}$ by $(p_z^i+1,p_z^i-1)$, and $q_z^{j,2}$ by $(q_z^j+1,q_z^j-1)$, for $1 \leq z \leq \frac{l+s-2}{2}$, $1 \leq i \leq x_z$, and $1 \leq j \leq y_z$, whenever $2r_{2z+1}> 2r_{2z}$; and then replacing $(2r_{1}, 0)$ by $(2r_{1}-1, 1)$, $p_0^{i,2}$ by $(p_0^i+1,p_0^i-1)$, and $q_0^{j,2}$ by $(q_0^j+1,q_0^j-1)$, for $1 \leq i \leq x_0$ and $1 \leq j \leq y_0$, if $2r_1\neq 0$.

On the other hand, we have 
\begin{align}\label{keylemma5equ23}
\begin{split}
&\left((\prod_{i=1}^u p_1^2)([(\prod_{j=1}^v q_j^2) \prod_{w=1}^{l+k} (2r_w)]^+)_{\SO_{2n^*+1}}\right)^{\SO_{2n+1}}\\
    &\qquad=\left((\prod_{i=1}^u p_1^2)[(\prod_{j=1}^v q_j^2) (2r_{l+k}+1)\prod_{w=1}^{l+k-1} (2r_w)]_{\SO_{2n^*+1}}\right)^{\SO_{2n+1}}. 
    \end{split}
\end{align}
Then $[(\prod_{j=1}^v q_j^2) (2r_{l+k}+1)\prod_{w=1}^{l+k-1} (2r_w)]_{\SO_{2n^*+1}}$ can be obtained from 
$$[(\prod_{j=1}^v q_j^2) (2r_{l+k}+1)\prod_{w=1}^{l+k-1} (2r_w)]$$
via replacing $(2r_{2z+1}, 2r_{2z})$ by $(2r_{2z+1}-1, 2r_{2z}+1)$ and $q_z^{j,2}$ by $(q_z^j+1,q_z^j-1)$, for $1 \leq z \leq \frac{l+s-2}{2}$ and $1 \leq j \leq y_z$, whenever $2r_{2z+1}> 2r_{2z}$; and then  replacing $(2r_{1}, 0)$ by $(2r_{1}-1, 1)$ and  $q_0^{j,2}$ by $(q_0^j+1,q_0^j-1)$ for $1 \leq j \leq y_0$, if $2r_1\neq 0$.
And the partition 
\[
\left((\prod_{i=1}^u p_1^2)[(\prod_{j=1}^v q_j^2) (2r_{l+k}+1)\prod_{w=1}^{l+k-1} (2r_w)]_{\SO_{2n^*+1}}\right)^{\SO_{2n+1}}
\]
can be obtained from $[(\prod_{i=1}^u p_1^2)[(\prod_{j=1}^v q_j^2) (2r_{l+k}+1)\prod_{w=1}^{l+k-1} (2r_w)]_{\SO_{2n^*+1}}]$ via replacing $p_z^{i,2}$ by $(p_z^i+1,p_z^i-1)$ for $1 \leq z \leq \frac{l+s-2}{2}$ and $1 \leq i \leq x_z$, whenever $2r_{2z+1}> 2r_{2z}$; and then  replacing $p_0^{i,2}$ by $(p_0^i+1,p_0^i-1)$ for $1 \leq i \leq x_0$, if $2r_1\neq 0$.
Hence \eqref{keylemma5equ22} holds in this case.

\quad

{\bf Case (1-b):} $q_v \leq 2r_{l+k}$, $p_u > 2r_{l+k}$. We have 
\begin{align*}
\begin{split}
&\left([(\prod_{i=1}^u p_1^2)(\prod_{j=1}^v q_j^2) \prod_{w=1}^{l+k} (2r_w)]^+\right)_{\SO_{2n+1}}\\
&\qquad=\left[(p_u+1)p_u(\prod_{i=1}^{u-1} p_1^2)(\prod_{j=1}^v q_j^2) \prod_{w=1}^{l+k} (2r_w)\right]_{\SO_{2n+1}}.
\end{split}
\end{align*}
To carry out the $\SO_{2n+1}$-collapse, we also need to list all the different even $p_i$'s between $p_u$ and $2r_{l+k}$ as 
\begin{align*}
   &\, p_u > p_{l+k}^1 > p_{l+k}^2 > \cdots > p_{l+k}^{x_{l+k}} > 2r_{l+k}.
\end{align*}
The collapse $[(\prod_{i=1}^u p_1^2)(\prod_{j=1}^v q_j^2) (2r_{l+k}+1)\prod_{w=1}^{l+k-1} (2r_w)]_{\SO_{2n+1}}$ can be obtained from 
\[
[(\prod_{i=1}^u p_1^2)(\prod_{j=1}^v q_j^2) (2r_{l+k}+1)\prod_{w=1}^{l+k-1} (2r_w)]
\]
via  replacing $(p_u, 2r_{l+k})$ by $(p_u-1, 2r_{l+k}+1)$ if $p_u$ is even, $(p_u+1, 2r_{l+k})$ by $(p_u, 2r_{l+k}+1)$ if $p_u$ is odd, and $p_{l+k}^{i,2}$ by $(p_{l+k}^i+1,p_{l+k}^i-1)$, for $1 \leq i \leq x_{l+k}$; and replacing $(2r_{2z+1}, 2r_{2z})$ by $(2r_{2z+1}-1, 2r_{2z}+1)$, $p_z^{i,2}$ by $(p_z^i+1,p_z^i-1)$, and $q_z^{j,2}$ by $(q_z^j+1,q_z^j-1)$, for $1 \leq z \leq \frac{l+s-2}{2}$, $1 \leq i \leq x_z$, and $1 \leq j \leq y_z$, whenever $2r_{2z+1}> 2r_{2z}$; and finally replacing $(2r_{1}, 0)$ by $(2r_{1}-1, 1)$, $p_0^{i,2}$ by $(p_0^i+1,p_0^i-1)$, and $q_0^{j,2}$ by $(q_0^j+1,q_0^j-1)$, for $1 \leq i \leq x_0$ and  $1 \leq j \leq y_0$, if $2r_1\neq 0$.

On the other hand, 
we still have \eqref{keylemma5equ23}.
We obtain the partition $[(\prod_{j=1}^v q_j^2) (2r_{l+k}+1)\prod_{w=1}^{l+k-1} (2r_w)]_{\SO_{2n^*+1}}$ from 
$$[(\prod_{j=1}^v q_j^2) (2r_{l+k}+1)\prod_{w=1}^{l+k-1} (2r_w)]$$
via replacing  
$(2r_{2z+1}, 2r_{2z})$ by $(2r_{2z+1}-1, 2r_{2z}+1)$ and $q_z^{j,2}$ by $(q_z^j+1,q_z^j-1)$, for $1 \leq z \leq \frac{l+s-2}{2}$ and $1 \leq j \leq y_z$, whenever $2r_{2z+1}> 2r_{2z}$; and replacing $(2r_{1}, 0)$ by $(2r_{1}-1, 1)$ and  $q_0^{j,2}$ by $(q_0^j+1,q_0^j-1)$, for $1 \leq j \leq y_0$, if $2r_1\neq 0$.
And the partition 
\[
\left((\prod_{i=1}^u p_1^2)[(\prod_{j=1}^v q_j^2) (2r_{l+k}+1)\prod_{w=1}^{l+k-1} (2r_w)]_{\SO_{2n^*+1}}\right)^{\SO_{2n+1}}
\]
can be obtained from $[(\prod_{i=1}^u p_1^2)[(\prod_{j=1}^v q_j^2) (2r_{l+k}+1)\prod_{w=1}^{l+k-1} (2r_w)]_{\SO_{2n^*+1}}]$ via replacing
$p_u^2$ by $(p_u+1, p_u-1)$ if $p_u$ is even, $p_{l+k}^{i,2}$ by $(p_{l+k}^i+1,p_{l+k}^i-1)$, for $1 \leq i \leq x_{l+k}$; and replacing $p_z^{i,2}$ by $(p_z^i+1,p_z^i-1)$, for $1 \leq z \leq \frac{l+s-2}{2}$, $1 \leq i \leq x_z$, whenever $2r_{2z+1}> 2r_{2z}$; and finally replacing $p_0^{i,2}$ by $(p_0^i+1,p_0^i-1)$, $1 \leq i \leq x_0$, if $2r_1\neq 0$.
Hence \eqref{keylemma5equ22} holds in this case.

\quad

{\bf Case (2-a):} $q_v > 2r_{l+k}$, $p_u \leq q_v$. We have 
\begin{align*}
\begin{split}
&\left([(\prod_{i=1}^u p_1^2)(\prod_{j=1}^v q_j^2) \prod_{w=1}^{l+k} (2r_w)]^+\right)_{\SO_{2n+1}}\\
&\qquad= \left[(\prod_{i=1}^u p_1^2)(q_v+1)q_v(\prod_{j=1}^{v-1} q_j^2) \prod_{w=1}^{l+k} (2r_w)\right]_{\SO_{2n+1}}.
\end{split}
\end{align*}
To carry out the $\SO_{2n+1}$-collapse, we also need to list all the different even $p_i$'s and $q_j$'s between $q_v$ and $2r_{l+k}$ as 
\begin{align*}
   &\, q_v > p_{l+k}^1 > p_{l+k}^2 > \cdots > p_{l+k}^{x_{l+k}} > 2r_{l+k},\\
      &\, q_v > q_{l+k}^1 > q_{l+k}^2 > \cdots > q_{l+k}^{y_{l+k}} > 2r_{l+k}.
\end{align*}

The collapse 
\[
\left[(\prod_{i=1}^u p_1^2)(q_v+1)q_v(\prod_{j=1}^{v-1} q_j^2) \prod_{w=1}^{l+k} (2r_w)\right]_{\SO_{2n+1}}
\]
can be obtained from 
\[
[(\prod_{i=1}^u p_1^2)(q_v+1)q_v(\prod_{j=1}^{v-1} q_j^2) \prod_{w=1}^{l+k} (2r_w)]
\]
via  replacing $(q_v, 2r_{l+k})$ by $(q_v-1, 2r_{l+k}+1)$ if $q_v$ is even, $(q_v+1, 2r_{l+k})$ by $(q_v, 2r_{l+k}+1)$ if $q_v$ is odd, $p_{l+k}^{i,2}$ by $(p_{l+k}^i+1,p_{l+k}^i-1)$, and $q_{l+k}^{j,2}$ by $(q_{l+k}^j+1,q_{l+k}^j-1)$, for $1 \leq i \leq x_{l+k}$ and $1 \leq j \leq y_{l+k}$; and replacing $(2r_{2z+1}, 2r_{2z})$ by $(2r_{2z+1}-1, 2r_{2z}+1)$, $p_z^{i,2}$ by $(p_z^i+1,p_z^i-1)$, and $q_z^{j,2}$ by $(q_z^j+1,q_z^j-1)$, for $1 \leq z \leq \frac{l+s-2}{2}$, $1 \leq i \leq x_z$, and $1 \leq j \leq y_z$, whenever $2r_{2z+1}> 2r_{2z}$; and finally replacing $(2r_{1}, 0)$ by $(2r_{1}-1, 1)$, $p_0^{i,2}$ by $(p_0^i+1,p_0^i-1)$, and $q_0^{j,2}$ by $(q_0^j+1,q_0^j-1)$, for $1 \leq i \leq x_0$ and $1 \leq j \leq y_0$, if $2r_1\neq 0$.

On the other hand, we have 
\begin{align}\label{keylemma5equ24}
\begin{split}
&\left((\prod_{i=1}^u p_1^2)([(\prod_{j=1}^v q_j^2) \prod_{w=1}^{l+k} (2r_w)]^+)_{\SO_{2n^*+1}}\right)^{\SO_{2n+1}}\\
    &\qquad=\left((\prod_{i=1}^u p_1^2)[(q_v+1)q_v(\prod_{j=1}^{v-1} q_j^2) \prod_{w=1}^{l+k} (2r_w)]_{\SO_{2n^*+1}}\right)^{\SO_{2n+1}}. 
    \end{split}
\end{align}
Then $[(q_v+1)q_v(\prod_{j=1}^{v-1} q_j^2) \prod_{w=1}^{l+k} (2r_w)]_{\SO_{2n^*+1}}$ can be obtained from 
$$[(q_v+1)q_v(\prod_{j=1}^{v-1} q_j^2) \prod_{w=1}^{l+k} (2r_w)]$$
via replacing $(q_v, 2r_{l+k})$ by $(q_v-1, 2r_{l+k}+1)$ if $q_v$ is even, $(q_v+1, 2r_{l+k})$ by $(q_v, 2r_{l+k}+1)$ if $q_v$ is odd, and $q_{l+k}^{j,2}$ by $(q_{l+k}^j+1,q_{l+k}^j-1)$, for $1 \leq j \leq y_{l+k}$; and  replacing 
$(2r_{2z+1}, 2r_{2z})$ by $(2r_{2z+1}-1, 2r_{2z}+1)$ and $q_z^{j,2}$ by $(q_z^j+1,q_z^j-1)$, for $1 \leq z \leq \frac{l+s-2}{2}$ and $1 \leq j \leq y_z$, whenever $2r_{2z+1}> 2r_{2z}$; and finally replacing $(2r_{1}, 0)$ by $(2r_{1}-1, 1)$ and  $q_0^{j,2}$ by $(q_0^j+1,q_0^j-1)$ for $1 \leq j \leq y_0$, if $2r_1\neq 0$.
And the partition 
\[
\left((\prod_{i=1}^u p_1^2)[(q_v+1)q_v(\prod_{j=1}^{v-1} q_j^2) \prod_{w=1}^{l+k} (2r_w)]_{\SO_{2n^*+1}}\right)^{\SO_{2n+1}}
\]
can be obtained from $(\prod_{i=1}^u p_1^2)[(q_v+1)q_v(\prod_{j=1}^{v-1} q_j^2) \prod_{w=1}^{l+k} (2r_w)]_{\SO_{2n^*+1}}$ via replacing
$p_u^2$ by $(p_u+1, p_u-1)$ if $p_u$ is even, and $p_{l+k}^{i,2}$ by $(p_{l+k}^i+1,p_{l+k}^i-1)$, for $1 \leq i \leq x_{l+k}$; and  replacing  $p_z^{i,2}$ by $(p_z^i+1,p_z^i-1)$ for $1 \leq z \leq \frac{l+s-2}{2}$ and $1 \leq i \leq x_z$, whenever $2r_{2z+1}> 2r_{2z}$; and finally replacing $p_0^{i,2}$ by $(p_0^i+1,p_0^i-1)$ for $1 \leq i \leq x_0$, if $2r_1\neq 0$.
Hence \eqref{keylemma5equ22} still holds in this case.

\quad

{\bf Case (2-b):} $q_v > 2r_{l+k}$, $p_u > q_v$. We have 
\begin{align*}
\begin{split}
&\left([(\prod_{i=1}^u p_1^2)(\prod_{j=1}^v q_j^2) \prod_{w=1}^{l+k} (2r_w)]^+\right)_{\SO_{2n+1}}\\
&\qquad \qquad = \left[(p_u+1)p_u(\prod_{i=1}^{u-1} p_1^2)(\prod_{j=1}^{v} q_j^2) \prod_{w=1}^{l+k} (2r_w)\right]_{\SO_{2n+1}}.
\end{split}
\end{align*}
To carry out the $\SO_{2n+1}$-collapse, we also need to list all the different even $p_i$'s and $q_j$'s between $p_u$ and $2r_{l+k}$ as 
\begin{align*}
   &\, p_u > p_{l+k}^1 > p_{l+k}^2 > \cdots > p_{l+k}^{x_{l+k}} > 2r_{l+k},\\
      &\, p_u > q_{l+k}^1 > q_{l+k}^2 > \cdots > q_{l+k}^{y_{l+k}} > 2r_{l+k}.
\end{align*}

The collapse $[(p_u+1)p_u(\prod_{i=1}^{u-1} p_1^2)(\prod_{j=1}^{v} q_j^2) \prod_{w=1}^{l+k} (2r_w)]_{\SO_{2n+1}}$ can be obtained from 
\[
[(p_u+1)p_u(\prod_{i=1}^{u-1} p_1^2)(\prod_{j=1}^{v} q_j^2) \prod_{w=1}^{l+k} (2r_w)]
\]
via replacing $(p_u, 2r_{l+k})$ by $(p_u-1, 2r_{l+k}+1)$ if $p_u$ is even, $(p_u+1, 2r_{l+k})$ by $(p_u, 2r_{l+k}+1)$ if $p_u$ is odd, $p_{l+k}^{i,2}$ by $(p_{l+k}^i+1,p_{l+k}^i-1)$, and  $q_{l+k}^{j,2}$ by $(q_{l+k}^j+1,q_{l+k}^j-1)$, for $1 \leq i \leq x_{l+k}$ and $1 \leq j \leq y_{l+k}$; and replacing $(2r_{2z+1}, 2r_{2z})$ by $(2r_{2z+1}-1, 2r_{2z}+1)$, $p_z^{i,2}$ by $(p_z^i+1,p_z^i-1)$, and $q_z^{j,2}$ by $(q_z^j+1,q_z^j-1)$, for $1 \leq z \leq \frac{l+s-2}{2}$, $1 \leq i \leq x_z$, and $1 \leq j \leq y_z$, whenever $2r_{2z+1}> 2r_{2z}$; and finally replacing $(2r_{1}, 0)$ by $(2r_{1}-1, 1)$, $p_0^{i,2}$ by $(p_0^i+1,p_0^i-1)$, and $q_0^{j,2}$ by $(q_0^j+1,q_0^j-1)$, for  $1 \leq i \leq x_0$ and $1 \leq j \leq y_0$, if $2r_1\neq 0$.

On the other hand, 
we still have \eqref{keylemma5equ24}. 
We obtain the partition 
$$
\left[(q_v+1)q_v(\prod_{j=1}^{v-1} q_j^2) \prod_{w=1}^{l+k} (2r_w)\right]_{\SO_{2n^*+1}}
$$ 
from 
$[(q_v+1)q_v(\prod_{j=1}^{v-1} q_j^2) \prod_{w=1}^{l+k} (2r_w)]$
via replacing $(q_v, 2r_{l+k})$ by $(q_v-1, 2r_{l+k}+1)$ if $q_v$ is even, $(q_v+1, 2r_{l+k})$ by $(q_v, 2r_{l+k}+1)$ if $q_v$ is odd, and $q_{l+k}^{j,2}$ by $(q_{l+k}^j+1,q_{l+k}^j-1)$ if $q_{l+k}^j\neq q_v$, for $1 \leq j \leq y_{l+k}$; and replacing 
$(2r_{2z+1}, 2r_{2z})$ by $(2r_{2z+1}-1, 2r_{2z}+1)$ and $q_z^{j,2}$ by $(q_z^j+1,q_z^j-1)$, for $1 \leq z \leq \frac{l+s-2}{2}$ and $1 \leq j \leq y_z$, whenever $2r_{2z+1}> 2r_{2z}$; and finally replacing $(2r_{1}, 0)$ by $(2r_{1}-1, 1)$ and  $q_0^{j,2}$ by $(q_0^j+1,q_0^j-1)$, for $1 \leq j \leq y_0$, if $2r_1\neq 0$.
And the partition 
\[
\left((\prod_{i=1}^u p_1^2)[(q_v+1)q_v(\prod_{j=1}^{v-1} q_j^2) \prod_{w=1}^{l+k} (2r_w)]_{\SO_{2n^*+1}}\right)^{\SO_{2n+1}}
\]
can be obtained from $(\prod_{i=1}^u p_1^2)[(q_v+1)q_v(\prod_{j=1}^{v-1} q_j^2) \prod_{w=1}^{l+k} (2r_w)]_{\SO_{2n^*+1}}$ via replacing
$p_u^2$ by $(p_u+1, p_u-1)$ if $p_u$ is even and $p_{l+k}^{i,2}$ by $(p_{l+k}^i+1,p_{l+k}^i-1)$, for $1 \leq i \leq x_{l+k}$; and  replacing  $p_z^{i,2}$ by $(p_z^i+1,p_z^i-1)$ for $1 \leq z \leq \frac{l+s-2}{2}$, $1 \leq i \leq x_z$, whenever $2r_{2z+1}> 2r_{2z}$; and finally replacing $p_0^{i,2}$ by $(p_0^i+1,p_0^i-1)$, for $1 \leq i \leq x_0$, if $2r_1\neq 0$.
Hence \eqref{keylemma5equ22} holds in this case. 

The proof of Lemma \ref{keylemma5} has been completed for $G_n=\SO_{2n+1}$.

\subsection{Proof of Lemma \ref{keylemma5}, $G_n=\RO_{2n}$}

By similar arguments as in the proof of the $\RO_{2n}$-case of Lemma \ref{keylemma2}, we only need to show that 
\begin{align}\label{keylemma5equ31}
\begin{split}
&\left([(\prod_{(\chi,m,\alpha) \in \textbf{e}} m^2)(\prod_{j=1}^t n_j^2) (\prod_{i=1}^l (2m_i+1)\prod_{s=1}^k (2n_s+1))]^{t}\right)_{\RO_{2n}}\\
&\quad=\left((\prod_{(\chi, m, \alpha) \in \textbf{e}} m^2)^t+ ([(\prod_{j=1}^t n_j^2) (\prod_{i=1}^l (2m_i+1)\prod_{s=1}^k (2n_s+1))]^{t})_{\RO_{2n^*}}\right)_{\RO_{2n}}.
\end{split}
\end{align}
For any given partition $\ul{p}=[p_r \cdots p_1]$ with $p_r \geq \cdots \geq p_1$, recall that $\ul{p}^+=[(p_r+1) \cdots p_1]$ and  $\ul{p}^-=[p_r \cdots (p_1-1)]$.
Rewrite the partition $[\prod_{(\chi, m, \alpha) \in \textbf{e}} m^2]$ as
$[p_u^2p_{u-1}^2\cdots p_1^2]$ with $p_u \geq p_{u-1} \geq \cdots \geq p_1$; and the partition 
$[\prod_{i=1}^t n_i^2]$ as
$[q_{v}^2q_{v-1}^2\cdots q_1^2]$ with $q_{v} \geq q_{v-1} \geq \cdots \geq q_1$.
And rewrite the partition $[(\prod_{i=1}^l (2m_i+1)\prod_{s=1}^k (2n_s+1))]$ as $[\prod_{w=1}^{l+k} (2r_w+1)]$
with $r_{l+k} \geq r_{l+k-1} \geq \cdots \geq r_{1} > 0$. Then, \eqref{keylemma5equ31} becomes
\begin{align}\label{keylemma5equ32}
\begin{split}
&\left([(\prod_{i=1}^u p_i^2)(\prod_{j=1}^v q_j^2) \prod_{w=1}^{l+k} (2r_w+1)]^{t}\right)_{\RO_{2n}}\\
&\qquad= \left((\prod_{i=1}^u p_i^2)^t+([(\prod_{j=1}^v q_j^2) \prod_{w=1}^{l+k} (2r_w+1)]^t)_{\RO_{2n^*}}\right)_{\RO_{2n}}.
\end{split}
\end{align}

As in the proof of the $\RO_{2n}$-case of Lemma \ref{keylemma2}, it is easy to see that
\begin{align*}
[(\prod_{i=1}^u p_i^2)(\prod_{j=1}^v q_j^2) \prod_{w=1}^{l+k} (2r_w+1)]^{t}
   = [(\prod_{i=1}^u p_i^2)(\prod_{j=1}^v q_j^2)]^t + ([\prod_{w=1}^{l+k}(2r_w+1)])^t 
\end{align*}
is a partition of the following form
$$[p_{l+k}^1 \cdots p_{l+k}^{2r_1+1} (\prod_{j=1}^{{l+k}-1} p_{j}^1\cdots p_{j}^{2r_{{l+k}+1-j}-2r_{{l+k}-j}})p_0^1 \cdots p_0^{r_0}],$$
where $p_{l+k}^{i}$ with $1\leq i \leq 2r_1+1$, $p_{2j}^i$ with  $1\leq i \leq 2r_{{l+k}+1-2j}-2r_{{l+k}-2j}$ and $1\leq j \leq \frac{{l+k}-2}{2}$, and $p_0^k$ with $1 \leq k \leq r_0$, 
are all even; and $p_{2j+1}^i$ with $1\leq i \leq 2r_{{l+k}-2j}-2r_{{l+k}-2j-1}$ and $0\leq j \leq \frac{{l+k}-2}{2}$ are all odd; and finally 
$p_{l+k}^1 \geq \cdots \geq p_{l+k}^{2r_1+1}$, $p_j^1 \geq \cdots \geq p_j^{2r_{{l+k}+1-j}-2r_{{l+k}-j}} \geq p_{j-1}^1$
with $1 \leq j \leq {l+k}-1$, and $p_0^1 \geq \cdots \geq p_0^{r_0}>0$. 

Following the recipe on carrying out the $\RO_{2n}$-collapse (\cite[Lemma 6.3.8]{CM93}), we obtain that 
$([(\prod_{i=1}^u p_i^2)(\prod_{j=1}^v q_j^2) \prod_{w=1}^{l+k} (2r_w+1)]^{t})_{\RO_{2n}}$ is equal to 
\begin{align}\label{eventotal}
\begin{split}
    &[(p_{l+k}^1 \cdots p_{l+k}^{2r_1})_O (p_{l+k}^{2r_1+1}-1)\prod_{j=0}^{\frac{{l+k}-2}{2}} p_{2j+1}^1\cdots p_{2j+1}^{2r_{{l+k}-2j}-2r_{{l+k}-2j-1}}\\
    &\quad\prod_{j=1}^{\frac{{l+k}-2}{2}}(p_{2j}^1+1)(p_{2j}^2\cdots p_{2j}^{2r_{{l+k}+1-2j}-2r_{{l+k}-2j}-1})_O(p_{2j}^{2r_{{l+k}+1-2j}-2r_{{l+k}-2j}}-1)\\
    &\quad\quad(p_0^1+1)(p_0^2 \cdots p_0^{r_0})_O].
    \end{split}
\end{align}
Similarly, we have 
\begin{align*}
   [(\prod_{j=1}^v q_j^2) \prod_{w=1}^{l+k} (2r_w+1)]^{t}
   =  [(\prod_{j=1}^v q_j^2)]^t + [\prod_{w=1}^{l+k}(2r_w+1)]^t \end{align*}
is a partition of the following form
$$[q_{l+k}^1 \cdots q_{l+k}^{2r_1+1} (\prod_{j=1}^{{l+k}-1} q_{j}^1\cdots q_{j}^{2r_{{l+k}+1-j}-2r_{{l+k}-j}})q_0^1 \cdots q_0^{r_0}],$$
where $q_{l+k}^{i}$ with $1\leq i \leq 2r_1+1$, $q_{2j}^i$ with  $1\leq i \leq 2r_{{l+k}+1-2j}-2r_{{l+k}-2j}$ and $1\leq j \leq \frac{{l+k}-2}{2}$, and $q_0^k$ with  $1 \leq k \leq r_0$, 
are all even; and $q_{2j+1}^i$ with $1\leq i \leq 2r_{{l+k}-2j}-2r_{{l+k}-2j-1}$ and $0\leq j \leq \frac{{l+k}-2}{2}$ 
are all odd; and finally $q_{l+k}^1 \geq \cdots \geq q_{l+k}^{2r_1+1}$, $q_j^1 \geq \cdots \geq q_j^{2r_{{l+k}+1-j}-2r_{{l+k}-j}} \geq q_{j-1}^1$ with $1 \leq j \leq {l+k}-1$, and $q_0^1 \geq \cdots \geq q_0^{r_0} \geq 0$. (Note that we are adding $0$'s at the end of the partition if necessary.)

Following the recipe on carrying out the $\RO_{2n^*}$-collapse (\cite[Lemma 6.3.8]{CM93}), we obtain that the partition 
$([(\prod_{j=1}^v q_j^2) \prod_{w=1}^{l+k} (2r_w+1)]^{t})_{\RO_{2n^*}}$ is equal to 
\begin{align*}
    &[(q_{l+k}^1 \cdots q_{l+k}^{2r_1})_\RO (q_{l+k}^{2r_1+1}-1)\prod_{j=0}^{\frac{{l+k}-2}{2}} q_{2j+1}^1\cdots q_{2j+1}^{2r_{{l+k}-2j}-2r_{{l+k}-2j-1}}\\
    &\quad\prod_{j=1}^{\frac{{l+k}-2}{2}}(q_{2j}^1+1)(q_{2j}^2\cdots q_{2j}^{2r_{{l+k}+1-2j}-2r_{{l+k}-2j}-1})_\RO(q_{2j}^{2r_{{l+k}+1-2j}-2r_{{l+k}-2j}}-1)\\
    &\quad\quad(q_0^1+1)(q_0^2 \cdots q_0^{r_0})_\RO].
\end{align*}
Without loss of generality and adding $0$'s if necessary, we can assume that $(\prod_{i=1}^u p_i^2)^t$ is an even partition having the following form
$$[s_{l+k}^1 \cdots s_{l+k}^{2r_1+1} (\prod_{j=1}^{{l+k}-1} s_{j}^1\cdots s_{j}^{2r_{{l+k}+1-j}-2r_{{l+k}-j}})s_0^1 \cdots s_0^{r_0}].$$
Then the partition 
$(\prod_{i=1}^u p_i^2)^t+([(\prod_{j=1}^v q_j^2) \prod_{w=1}^{l+k} (2r_w+1)]^t)_{\RO_{2n^*}}$ is equal to 
\begin{align*}
   &[((s_{l+k}^1 \cdots s_{l+k}^{2r_1})+(q_{l+k}^1 \cdots q_{l+k}^{2r_1})_\RO) (t_{l+k}^{2r_1+1}-1)\prod_{j=0}^{\frac{{l+k}-2}{2}} t_{2j+1}^1\cdots t_{2j+1}^{2r_{{l+k}-2j}-2r_{{l+k}-2j-1}}\\
    &\quad\prod_{j=1}^{\frac{{l+k}-2}{2}}(t_{2j}^1+1)((s_{2j}^2\cdots s_{2j}^{2r_{{l+k}+1-2j}-2r_{{l+k}-2j}-1})+(q_{2j}^2\cdots q_{2j}^{2r_{{l+k}+1-2j}-2r_{{l+k}-2j}-1})_\RO)\\
    &\quad\quad (t_{2j}^{2r_{{l+k}+1-2j}-2r_{{l+k}-2j}}-1)(t_0^1+1)((s_0^2 \cdots s_0^{r_0})+(q_0^2 \cdots q_0^{r_0})_\RO)]，
\end{align*}
where all the $t$-term's are the summation of the corresponding $q$-terms and $s$-terms. It is clear that the $t$-terms are exactly the corresponding $p$-terms
in \eqref{eventotal}. 

Now, following the recipe on carrying out the $\RO_{2n^*}$-collapse (\cite[Lemma 6.3.8]{CM93}), we obtain that
the partition 
\[
\left((\prod_{i=1}^u p_i^2)^t+([(\prod_{j=1}^v q_j^2) \prod_{w=1}^{l+k} (2r_w+1)]^t)_{\RO_{2n^*}}\right)_{\RO_{2n}}
\]
is equal to 
\begin{align*}
  [((s_{l+k}^1 \cdots s_{l+k}^{2r_1})+(q_{l+k}^1 \cdots q_{l+k}^{2r_1})_\RO)_\RO (t_{l+k}^{2r_1+1}-1)\prod_{j=0}^{\frac{{l+k}-2}{2}} t_{2j+1}^1\cdots t_{2j+1}^{2r_{{l+k}-2j}-2r_{{l+k}-2j-1}}\\
\prod_{j=1}^{\frac{{l+k}-2}{2}}(t_{2j}^1+1)((s_{2j}^2\cdots s_{2j}^{2r_{{l+k}+1-2j}-2r_{{l+k}-2j}-1})+(q_{2j}^2\cdots q_{2j}^{2r_{{l+k}+1-2j}-2r_{{l+k}-2j}-1})_\RO)_\RO\\
(t_{2j}^{2r_{{l+k}+1-2j}-2r_{{l+k}-2j}}-1)(t_0^1+1)((s_0^2 \cdots s_0^{r_0})+(q_0^2 \cdots q_0^{r_0})_\RO)_\RO].
\end{align*}
Since $s$-terms are all even, by \cite[Lemma 3.1]{Ac03}, we have 
\begin{align*}
    ((s_{l+k}^1 \cdots s_{l+k}^{2r_1})+(q_{l+k}^1 \cdots q_{l+k}^{2r_1})_\RO)_\RO
    =&\, (s_{l+k}^1 \cdots s_{l+k}^{2r_1})_O+(q_{l+k}^1 \cdots q_{l+k}^{2r_1})_\RO\\
    =&\,((s_{l+k}^1 \cdots s_{l+k}^{2r_1})+(q_{l+k}^1 \cdots q_{l+k}^{2r_1}))_\RO\\
    =&\,(p_{l+k}^1 \cdots p_{l+k}^{2r_1})_\RO,
\end{align*}
\begin{align*}
    &((s_{2j}^2\cdots s_{2j}^{2r_{{l+k}+1-2j}-2r_{{l+k}-2j}-1})+(q_{2j}^2\cdots q_{2j}^{2r_{{l+k}+1-2j}-2r_{{l+k}-2j}-1})_\RO)_\RO\\
    &\qquad= (s_{2j}^2\cdots s_{2j}^{2r_{{l+k}+1-2j}-2r_{{l+k}-2j}-1})_\RO+(q_{2j}^2\cdots q_{2j}^{2r_{{l+k}+1-2j}-2r_{{l+k}-2j}-1})_\RO\\
    &\qquad\qquad=((s_{2j}^2\cdots s_{2j}^{2r_{{l+k}+1-2j}-2r_{{l+k}-2j}-1})+(q_{2j}^2\cdots q_{2j}^{2r_{{l+k}+1-2j}-2r_{{l+k}-2j}-1}))_\RO\\
    &\qquad\qquad\qquad=(p_{2j}^2\cdots p_{2j}^{2r_{{l+k}+1-2j}-2r_{{l+k}-2j}-1})_\RO,
\end{align*}
and
\begin{align*}
    ((s_0^2 \cdots s_0^{r_0})+(q_0^2 \cdots q_0^{r_0})_\RO)_\RO
    =&\, (s_0^2 \cdots s_0^{r_0})_O+(q_0^2 \cdots q_0^{r_0})_\RO\\
    =&\, ((s_0^2 \cdots s_0^{r_0})+(q_0^2 \cdots q_0^{r_0}))_\RO\\
    =&\, (p_0^2 \cdots p_0^{r_0})_\RO.
\end{align*}
Hence the partition 
$((\prod_{i=1}^u p_i^2)^t+([(\prod_{j=1}^v q_j^2) \prod_{w=1}^{l+k} (2r_w+1)]^t)_{O_{2n^*}})_{\RO_{2n}}$ is equal to 
\begin{align*}
  &[(p_{l+k}^1 \cdots p_{l+k}^{2r_1})_\RO (p_{l+k}^{2r_1+1}-1)\prod_{j=0}^{\frac{{l+k}-2}{2}} p_{2j+1}^1\cdots p_{2j+1}^{2r_{{l+k}-2j}-2r_{{l+k}-2j-1}}\\
    &\qquad\prod_{j=1}^{\frac{{l+k}-2}{2}}(p_{2j}^1+1)(p_{2j}^2\cdots p_{2j}^{2r_{{l+k}+1-2j}-2r_{{l+k}-2j}-1})_\RO\\
    &\qquad\qquad(p_{2j}^{2r_{{l+k}+1-2j}-2r_{{l+k}-2j}}-1)(p_0^1+1)(p_0^2 \cdots p_0^{r_0})_\RO],
\end{align*}
which is exactly equal to $([(\prod_{i=1}^u p_i^2)(\prod_{j=1}^v q_j^2) \prod_{w=1}^{l+k} (2r_w+1)]^{t})_{\RO_{2n}}$ by \eqref{eventotal}. 
Therefore, we have shown \eqref{keylemma5equ31}, hence complete the proof of $G_n=\RO_{2n}$-case of Lemma \ref{keylemma5}.


\end{document}